\documentclass[11pt]{article} 
\usepackage{macros}

\newcommand{\lambdacsup}{\ensuremath \bar \lambda_C}
\newcommand{\lambdacinf}{\ensuremath \ubar \lambda_C}

\newcommand{\gammacsup}{\ensuremath \bar \gamma_C}

\newcommand{\lambdak}{\ensuremath {\lambda_K}}
\newcommand{\lambdaksup}{\ensuremath \bar \lambda_K}
\newcommand{\lambdakinf}{\ensuremath \ubar \lambda_K}

\newcommand{\gammak}{\ensuremath {\gamma_K}}
\newcommand{\gammaksup}{\ensuremath {\bar \gamma_K}}
\newcommand{\gammakinf}{\ensuremath {\ubar \gamma_K}}

\newcommand{\bp}{{\scalebox{0.6} {\ensuremath \boxplus}}}
\newcommand{\bc}{\ensuremath \oplus}

\begin{document}

\begin{center}

	{\bf{\LARGE{Reverse Euclidean and Gaussian isoperimetric\\ \vspace{2.5mm} inequalities for parallel sets
	  with applications}}}

	\vspace*{.25in}

 	\begin{tabular}{cc}
 		 & {\large{Varun Jog}} \\
 		& {\large{\texttt{vjog@wisc.edu}}} 
 			\end{tabular}
 \begin{center}
 Department of Electrical \& Computer Engineering\\
 University of Wisconsin-Madison
  \end{center}

	\vspace*{.2in}

August 2020

	\vspace*{.2in}

\end{center}


\begin{abstract}
The $r$-parallel set of a measurable set $A \subseteq \real^d$ is the set of all points whose distance from $A$ is at most $r$. In this paper, we show that the surface area of an $r$-parallel set in $\real^d$ with volume at most $V$ is upper-bounded by $e^{\Theta(d)}V/r$, whereas its Gaussian surface area is upper-bounded by $\max(e^{\Theta(d)}, e^{\Theta(d)}/r)$. We also derive a reverse form of the Brunn-Minkowski inequality for $r$-parallel sets, and as an aside a reverse entropy power inequality for Gaussian-smoothed random variables. We apply our results to two problems in theoretical machine learning: (1) bounding the computational complexity of learning $r$-parallel sets under a Gaussian distribution; and (2) bounding the sample complexity of estimating robust risk, which is a notion of risk in the adversarial machine learning literature that is analogous to the Bayes risk in hypothesis testing.\end{abstract} 

\section{Introduction}

The isoperimetric problem in $\real^d$ poses the following question: What is the minimum surface area of a set in $\real^n$ with a given volume? Equivalently, what is the maximum volume of a set in $\real^n$ with a given surface area? It is well known that Euclidean balls are the unique extremal sets for both formulations; i.e., the following inequality holds for all sets with surface area $S$ and volume $V$, with equality if and only if the set is a Euclidean ball:
\begin{align*}
S^d \geq d^d \omega_d V^{d-1},
\end{align*}
where $\omega_d$ is the volume of the unit $\ell_2$-ball in $\real^d$. Although intuitive, this inequality is non-trivial to prove and holds in astonishing generality. Indeed, with the right definition of ``surface area,'' the isoperimetric inequality holds for all measurable sets, with no regularity conditions on the boundary~\cite{Fed14}. The volume of a measurable set $A \subseteq \real^n$ is its Lebesgue measure $\lambda(A)$. In this paper, we use the notion of \emph{Minkowski surface area}, defined as follows:

\begin{definition}\label{def: leb_area}
Let $A$ be a measurable set. Let $\cK_d$ denote the set of centrally-symmetric, bounded, convex sets in $\real^d$. For $r > 0$, the $r$-parallel body of $A$ with respect to $K \in \cK_d$ is given by the Minkowski sum $A \oplus rK$. Define the $K$-surface area of $A$ as:
\begin{align}\label{eq: leb_area}
\lambdakinf(\del A) &\defn \liminf_{\delta \to 0} \frac{\lambda(A \oplus \delta K) - \lambda(A)}{\delta},\\
\lambdaksup(\del A) &\defn \limsup_{\delta \to 0} \frac{\lambda(A \oplus \delta K) - \lambda(A)}{\delta}.
\end{align}
We call the former the lower $K$-surface area and the latter the upper $K$-surface area. If $\lambdakinf(\del A) = \lambdaksup(\del A)$, we refer to the quantity as the $K$-surface area and denote it by $\lambdak(\del A)$. If $B$ is the unit $\ell_2$-ball, we drop the subscript $B$ and refer to the quantities $\lambda(\del A), \bar \lambda(\del A),$ and $\ubar \lambda(\del A)$ as the surface area, upper surface area, and lower surface area, respectively. 
\end{definition}
Using the above notation, we can rewrite the isoperimetric inequality as
\begin{align}\label{eq: leb_iso}
\ubar \lambda(\del A)^d \geq d^d \omega_d \lambda(A)^{d-1}.
\end{align}
We shall refer to inequality~\eqref{eq: leb_iso} as the \emph{Euclidean isoperimetric inequality}. It is possible to interpret ``volume'' using a measure other than the Lebesgue measure. A popular alternative is the standard Gaussian measure, which we denote by $\gamma$. 

\begin{definition}\label{def: gauss_area}
For $K \in \cK_d$, the lower- and upper-Gaussian $K$-surface areas of a measurable set $A$ are defined as follows:
\begin{align}\label{eq: gauss_area}
\gammakinf(\del A) \defn \liminf_{\delta \to 0} \frac{\gamma(A \oplus \delta K) - \gamma(A)}{\delta},\\
\gammaksup(\del A) \defn \limsup_{\delta \to 0} \frac{\gamma(A \oplus \delta K) - \gamma(A)}{\delta}.
\end{align}
If $\gammakinf(\del A) = \gammaksup(\del A)$, we refer to the quantity as the Gaussian $K$-surface area and denote it by $\gammak(\del A)$. If $B$ is the unit $\ell_2$-ball, we drop the subscript $B$ and refer to the quantities $\gamma(\del A), \bar \gamma(\del A),$ and $\ubar \gamma(\del A)$ as the Gaussian surface area, upper Gaussian surface area, and lower Gaussian surface area, respectively. 
\end{definition}
Analogous to the Euclidean isoperimetric inequality, Sudakov and Tsirel'son~\cite{SudTsi78} and Borel~\cite{Bor75} established the \emph{Gaussian isoperimetric inequality}. This inequality states that among all sets with a given Gaussian volume, halfspaces have the minimum possible Gaussian surface areas. It is worth noting that a halfspace has infinite Euclidean surface area but its Gaussian surface area is bounded above by the constant $\sqrt{1/2\pi}$.

The \emph{reverse} isoperimetric problem is the following: What is the maximum surface area of a set in $\real^d$ with fixed volume? Equivalently, what is the minimum volume of a set in $\real^d$ with a fixed surface area? A little reflection reveals that this question does not make sense as posed, for we can have sets such as spheres that have zero volume but arbitrarily large surface area. To make sense of the reverse isoperimetric problem, it is necessary to impose some regularity conditions on the class of sets being considered to prevent ``wiggliness'' of the boundary. 

Reverse isoperimetric inequalities are more easily described in the Gaussian setting than the Euclidean setting. Ball~\cite{Ball93} established a reverse Gaussian isoperimetric inequality for convex sets: The Gaussian surface area of any convex set $A \subseteq \real^d$ is bounded above by $4d^{1/4}$. Nazarov~\cite{Naz03} further refined Ball's bound and also showed that it is essentially tight by constructing a set with $\Theta(d^{1/4})$ Gaussian surfacearea. Generalizations of Ball~\cite{Ball93} and Nazarov~\cite{Naz03} for log-concave measures were obtained in Livshyts~\cite{Liv13, Liv14}. Klivans, O'Donnell, and Servedio~\cite{KliEtal08} established a link between the Gaussian surface areas of sets and the ability to learn them efficiently under the probably-approximately-correct (PAC) and agnostic learning models. Klivans et al.\ showed that sets with small Gaussian surface areas can be learned efficiently under the Gaussian distribution. Klivans et al.\ obtained bounds on the Gaussian surface areas of cones and balls, and Kane~\cite{Kane11} bounded the Gaussian surface areas sets obtained from thresholded polynomials of a fixed degree.

In the Euclidean setting, most existing work focuses on sets in $\real^2$ and $\real^3$ with some kind of curvature constraint on the boundary of the sets. Howard and Treibergs~\cite{HowTre95} showed that if the average curvature $\kappa$ of a curve in $\real^2$ satisfies $|\kappa| \leq 1$, and if the area enclosed by the curve is small enough, then a certain ``peanut shape'' has the largest perimeter for a fixed area. Gard~\cite{Gar12} extended this result to surfaces of revolution in $\real^3$. Pan, Tang, and Wang~\cite{PanEtal10} obtained a version of the reverse isoperimetric inequality for sets in $\real^2$ by lower-bounding the perimeter in terms of the area of the set as well as the the area of the locus of its curvature centers. The one result we were able to find that holds in higher dimensions is that of Chernov, Drach, and Tatarko~\cite{CheEtal19}, where the authors showed that for \emph{convex} sets satisfying a weak notion of curvature constraint called $\lambda$-concavity, the sausage body (Minkowski sum of a line segment and an $\ell_2$-ball) has the largest surface area for a fixed volume. Another result that holds for convex sets in arbitrary dimensions is that of Ball~\cite{Ball91}; however, it involves transforming the set via a volume-preserving linear map and thus cannot be compared to the above results.

In this paper, we take a different approach towards imposing regularity conditions. Our goal will be to study reverse isoperimetric inequalities for \emph{$r$-parallel sets} which are defined as follows:
\begin{definition}
Let $r > 0$, $d \geq 1$, and $K \in \cK_d$. A set $\tilde A \subseteq \real^d$ is called an $r$-parallel set with respect to $K$ if $\tilde A = A \oplus rK$ for some measurable set $A$. 
\end{definition}
Throughout this paper, we shall be concerned with only two sets $K$: the unit ball in the $\ell_2$-norm, and the unit ball in the $\ell_\infty$-norm. We shall denote them as follows:
\begin{align*}
B &= \{x \in \real^d ~|~ \|x\|_2 \leq 1\} \quad \text{ and}\\
C &= \{x \in \real^d ~|~ \|x\|_\infty \leq 1\} = [-1,1]^d.
\end{align*}
As $C/\sqrt d \subseteq B \subseteq C$, we have that 
\begin{align*}
\frac{\lambdacinf(\del A)}{\sqrt d} &\leq \ubar \lambda(\del A) \leq \lambdacinf(\del A) \quad \text{ and }\\
\frac{\lambdacsup(\del A)}{\sqrt d} &\leq \bar \lambda(\del A) \leq \lambdacsup(\del A).
\end{align*}
The same inequalities also hold for the Gaussian measure. It turns out that the factor of $\sqrt d$ does not play an important role in our results, which are essentially identical for both notions of surface areas.

The notation $A_{\oplus r}$ and $A_{\bp r}$ will be used to represent an $r$-parallel set of some measurable set $A$ with respect to $B$ and $C$, respectively. For $K \in \{B, C\}$, the intuition is that even if $A$ has a very wiggly boundary, the set $A \oplus rK$ will have a better-behaved boundary. It is clear, though, that the boundary of $A \oplus rK$ need not be twice-continuously differentiable, or even a union of finitely many such pieces. Moreover, the sets $A \oplus rK$ need not be convex. These observations preclude the possibility of directly using any of the reverse isoperimetric inequalities known in the literature. 

Parallel sets with respect to $B$ appear prominently in the context of quermassintegrals, intrinsic volumes, and Steiner's formula for the volume of the Minkowski sum of a convex set with a ball~\cite{Sch14}. Over the years, parallel sets of arbitrary closed sets have also been investigated and some regularity properties have been established in the process. The work most relevant to ours is Stacho~\cite{Sta76}, and we shall utilize several results and techniques from that paper in the course of our proofs. For now, we point out that for $r > 0$, Stacho~\cite{Sta76} showed that for $K \in \cK_d$, the Minkowski $K$-surface area of a bounded set $A \oplus rK$ can be calculated as the limit
\begin{align*}
\lambdak(\del (A \oplus rK)) = \lim_{\delta \to 0} \frac{\lambda(A \oplus (r+\delta)K) - \lambda(A \oplus rK)}{\delta},
\end{align*}
that is $\lambdaksup(\del (A \oplus rK)) = \lambdakinf(\del (A \oplus rK))$. Recent work by Hug, Last, and Weil~\cite{HugEtal04} and Rataj and Winter~\cite{RatWin10} has strengthened the results from Stacho~\cite{Sta76}. Hug et al.~\cite{HugEtal04} showed a local version of Steiner's formula for arbitrary closed sets, whereas Rataj and Winter~\cite{RatWin10} proved results concerning rectifiability of parallel sets and established relations between various notions of surface areas of parallel sets, including the Hausdorff measure of the boundary, the lower and upper Minkowski contents of the boundary, and Minkowski's surface area from Definition~\ref{def: leb_area}.

Another motivation for considering $r$-parallel sets comes from information theory. The information theoretic concepts of entropy and Fisher information have often been compared to the geometric concepts of volume and surface area~\cite{CosCov84}. A striking similarity exists between the definition of surface area in equation~\eqref{eq: leb_area} and de Bruijn's identity from information theory: Given a random vector $X$ on $\real^d$ and a standard normal random variable $Z$ that is independent of $X$, the Fisher information of $X$, denoted by $J(X)$, satisfies the relation
\begin{align*}
\frac{d}{dt} h(X+ \sqrt t Z) \Big |_{t=0} = \frac{J(X)}{2}.
\end{align*}
This means that
\begin{align}\label{eq: it_rip}
\frac{d}{dt} e^{h(X + \sqrt t Z)} \Big|_{t=0} = \frac{e^{h(X)}J(X)}{2}.
\end{align}
Thus, the Minkowski sum with a ball is replaced by a sum with independent Gaussian noise; volume is replaced by the exponential of the entropy; and surface area is replaced by a scaled version of the Fisher information. The analogous notion of $r$-parallel sets in information theory would be the set of all random vectors $X_r \defn X + \sqrt r Z$, which we call $r$-smoothed random variables. A version of the reverse isoperimetric inequality in information theory could be stated as: Given an $r$-smoothed random variable of a fixed entropy $h_0$, how large can its scaled-Fisher information be? Surprisingly, it is very easy to obtain such an upper bound. It is a well-known fact that Fisher information is a convex functional on the space of distributions~\cite{Zam98}, so $J(X_r) = J(X + \sqrt r Z) \leq J(\sqrt r Z) = d/r$. Thus, we conclude that $e^{h(X)}J(X)/2 \leq e^{h_0}d/2r$. Does a version of the reverse isoperimetric inequality exist for $r$-parallel sets in geometry?

This is precisely the question addressed in our paper. We study two problems of interest: (i) Is it possible to upper bound the surface area of an $r$-parallel set given a bound on its volume?; and (ii) is there a version of the reverse Gaussian isoperimetric inequality for $r$-parallel sets? Our result concerning (i) may be informally stated as follows:
\begin{result}[Formal statement in Theorem~\ref{thm: lebesgue_rip_volume}]\label{res: lebesgue}
Let $r > 0$. Then the following inequalities hold for some dimension-dependent constant $C_d = e^{\Theta(d)}$:
\begin{enumerate}
\item
If $\lambda(A_{\oplus r}) \leq V$, then $\lambda(\del A_{\oplus r}) \leq C_d \frac{V}{r}$.
\item
If $\lambda(A_{\bp r}) \leq V$, then $\lambda(\del A_{\bp r}) \leq C_d \frac{V}{r}$.
\end{enumerate} 
\end{result}
Observe that the bound increases as $r$ decreases, which is to be expected, since the sets $A_{\oplus r}$ and $A_{\bp r}$ have fewer restrictions on their boundaries. It is interesting to note that the $1/r$ dependence is the same as in the information theoretic reverse isoperimetric inequality in equation~\eqref{eq: it_rip}. Having proved the reverse Euclidean isoperimetric inequality, we use a proof technique from Ball~\cite{Ball93} to establish its analog for the Gaussian measure. Our result can be informally stated as follows:
\begin{result}[Formal statement in Theorem~\ref{thm: gaussian_rip}]
Let $r > 0$. Then the following bound holds for some dimension-dependent constant $C_d = e^{\Theta(d)}$:
\begin{enumerate}
\item
$\bar \gamma(\del A_{\oplus r}) \leq \max\left(C_d, \frac{C_d}{r} \right)$.
\item
$\bar \gamma(\del A_{\bp r}) \leq \max\left(C_d, \frac{C_d}{r} \right)$.
\end{enumerate}
\end{result}

Just as in the reverse isoperimetric inequality for convex sets in Ball~\cite{Ball93} and Nazarov~\cite{Naz03}, we do not need to impose any boundedness assumptions on $A_{\oplus r}$ or $A_{\bp r}$. 

We also provide two applications of the reverse Gaussian isoperimetric inequality to learning theory. First, we show that the machinery in Klivans et al.~\cite{KliEtal08} provides computational complexity bounds for learning $r$-parallel sets under the Gaussian distribution. Our second application concerns adversarial machine learning. Notions of robust risk, analogous to Bayes risk in standard hypothesis testing, have recently been proposed in the machine learning literature. Some recent work by Bhagoji, Cullina, and Mittal~\cite{BhaEtal19} and Pydi and Jog~\cite{PydiJog20} characterizes robust risk in terms of an optimal transport cost between the data distributions of two classes in a binary classification setting. We show that for Gaussian-smoothed data distributions, the Gaussian reverse isoperimetric inequality can be used to provide sample complexity bounds for estimating robust risk. 

The structure of this paper is as follows: In Section~\ref{sec: puzzle}, we present two puzzles in $\real^2$ whose solutions capture the essence of our proof. In Section~\ref{sec: leb} and Section~\ref{sec: gauss}, we prove the reverse isoperimetric inequalities in the Euclidean and Gaussian settings, respectively. In Section~\ref{sec: rev_bmi_epi} we prove versions of the reverse Brunn-Minkowski and the reverse entropy power inequality. In Section~\ref{sec: ml}, we describe applications to learning theory. Finally, we conclude the paper in Section~\ref{sec: end}.

\paragraph{Notation: } 
\begin{itemize}
\item
The unit ball in the $\ell_2$-norm and the $\ell_\infty$-norm are denoted by $B$ and $C$, respectively. If the dimension is not clear from context, we shall use $B_d$ and $C_d$. The Euclidean ball in $\real^d$ with center $x$ and radius $r$ is denoted by $B_d(x; r)$ (or $B(x; r)$, if the dimension is clear from context). The ball centered at the origin is denoted by $B_d(r)$ or $B(r)$. The notation $C_d(x; r)$, $C(x;r)$, $C_d(r)$, and $C(r)$ is defined analogously for $\ell_\infty$-balls.
\item
Given two measurable sets $A,B \subseteq \real^d$, their Minkowski sum is given by 
$$A \oplus B = \{a+b ~\mid~ a \in A, b\in B\}.$$
\item
We use the shorthand $A_{\oplus r} \defn A \oplus rB$ and $A_{\bp r} \defn A \oplus rC$. In Sections~\ref{sec: rev_bmi_epi} and~\ref{sec: ml}, we will only consider parallel sets of the form $A_{\bc r}$, since the results for $A_{\bp r}$ are identical. The shorthand $A_r \defn A_{\bc r}$ will be used in these sections.
\item
The volume of the $d$-dimensional unit ball is denoted by $\omega_d$ and its surface area is denoted by $\Omega_d$. The exact formulas are $\omega_d = \frac{\pi^{d/2}}{\Gamma(1 + d/2)}$, and $\Omega_d = d \omega_d$.
\item
The solid angle subtended by a set $S$ towards a point $x$ is denoted by $\Omega(S; x)$.
\item
The distance of a point $x$ from a set $A$ is $d(x, A) = \inf_{a \in A} d(x, a)$, where $d$ is a metric on $\real^d$. 
\item
$\| x \|_p$ indicates the $\ell_p$-norm in $\real^d$.
\item
$\mathbbm 1\{x \in A\}$ is the indicator function for the event $x \in A$.
\item
For $N \geq 1$, we use the notation $[ N ] \defn \{1, 2, \dots, N\}$.
\item
Given functions $f, g : \mathbb N \to \real$, we say that $f = \Theta(g)$ if there exist constants $c_1, c_2 > 0$ such that $c_1 g(n) \leq f(n) \leq c_2 g(n)$.
\end{itemize}

\section{Two puzzles in $\real^2$}\label{sec: puzzle}

In this section, we present two puzzles in $\real^2$, whose solutions neatly capture the main ideas in our approach. 

\paragraph{B-Puzzle:} Consider $N \geq 1$ points $x_i \in B(x_0; 1)$, for $1 \leq i \leq N$. Let $A = \{x_i \in \real^2 ~\mid~ 0 \leq i \leq N\}$. Show that the perimeter of $A_{\oplus 1} = A \oplus B(1)$ is no more than that of $B(x_0; 2)$; i.e.,
\begin{align}\label{eq: puzzle}
\lambda(\del A_{\oplus 1}) \leq \lambda(\del B(2)) = 4\pi.
\end{align}
Figure~\ref{fig: puzzle} shows an example of the set $A_{\oplus 1}$.
\begin{figure}
\begin{center}
\includegraphics[width = 3 in]{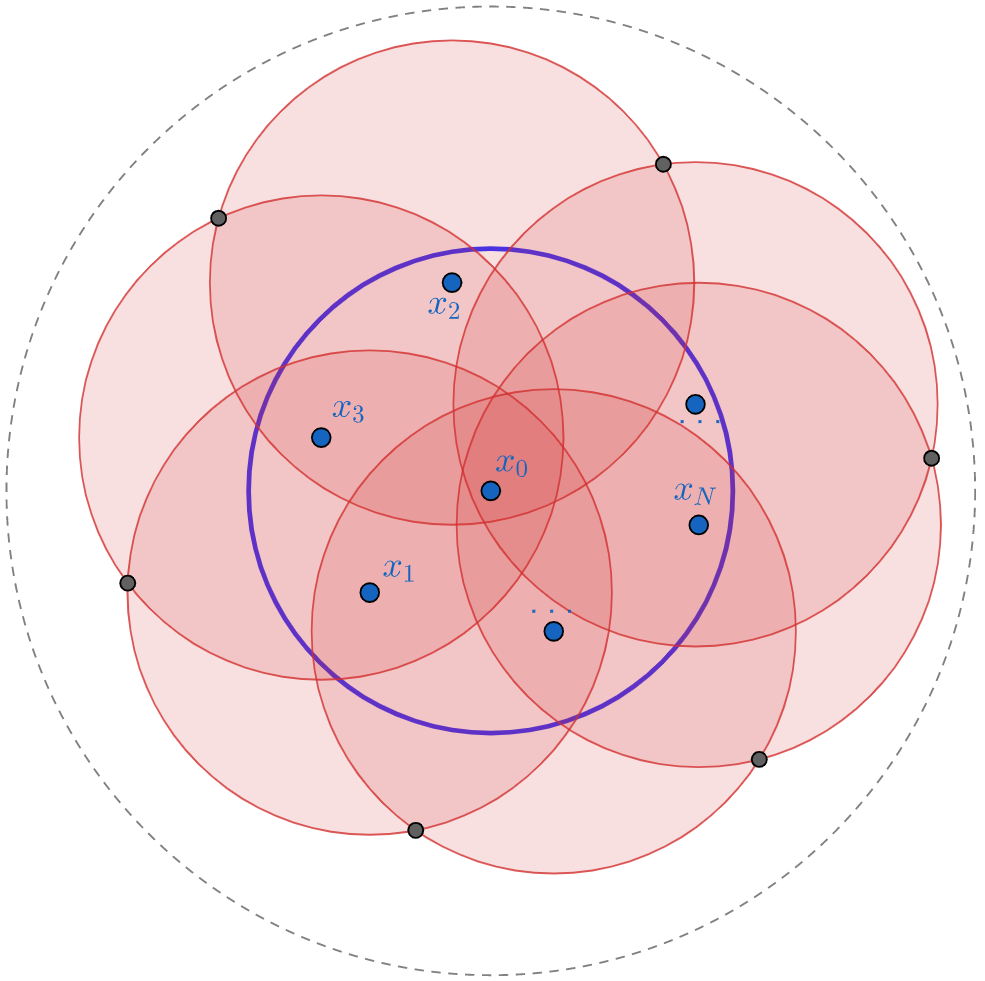}
\end{center}
\caption{Given $\{x_1, \dots, x_n\} \subseteq B(x_0; 1)$, the shaded region is the union of unit balls with centers at $x_i$, for $0 \leq i \leq N$. The problem is to upper-bound the perimeter of the shaded region with that of $B(x_0; 2)$.} \label{fig: puzzle}
\end{figure}

\paragraph{Solution:} A simple upper bound on the perimeter of $A_{\oplus 1}$ is $\sum_{i=0}^N \lambda(\del(B(x_i; 1))) = 2\pi(N+1)$; however, this bound becomes progressively weaker with increasing $N$. One may wonder whether equality is ever achieved in inequality~\eqref{eq: puzzle}, and a little reflection reveals that almost any arrangement of $N$ points on the circumference of $B(x_0; 1)$ gives equality. The only condition needed for the arrangement is that $\del A_{\oplus 1}$ contains no contribution from $\del B(x_0; 1)$. For instance, three points equally spaced on the perimeter of $B(x_0; 1)$ suffice. We make two observations:
\begin{itemize}
\item[(1)] The set $A_{\oplus 1}$ is \emph{star-shaped} from the point of view of $x_0$; i.e., any ray starting from $x_0$ intersects the boundary of $A_{\oplus 1}$ exactly once. Suppose this were not the case and a ray from $x_0$ were to intersect the boundary of $A_{\oplus 1}$ in two points $y_1$ and $y_2$, where we assume that $y_1$ is closer to $x_0$ than $y_2$. For $i \in \{1, 2\}$, the very fact that $y_i$ lies on the boundary of $A_{\oplus 1}$ means that $ B(y_i; 1) \cap B(x_0; 1)$ cannot contain any points $x_i$, except for one or more that lie on $\del B(y_i; 1) \cap B(x_0; 1)$. This immediately leads to a contradiction, since any point $x_i$ that lies on $\del B(y_2; 1) \cap B(x_0; 1)$ will lie within $B(y_1; 1) \cap B(x_0; 1)$, but not be on $\del B(y_1; 1) \cap B(x_0; 1)$.

\item[(2)] The boundary of $A_{\oplus 1}$ can be partitioned as $\del A_{\oplus 1} =   \cup_{i=0}^N (\del A_{\oplus 1})^i$, where $(\del A_{\oplus 1})^i = \cup_{i=0}^N \del B(x_i; 1) \cap \del A_{\oplus 1}$ is the arc of the circle $B(x_i; 1)$ that lies on $\del A_{\oplus 1}$. (Note that several of the sets $(\del A_{\oplus 1})^i$ may be empty.) The perimeter of $A_{\oplus 1}$ may be expressed as $\lambda(\del A_{\oplus 1}) = \sum_{i=1}^N \lambda((\del A_{\oplus 1})^i)$. Without loss of generality, suppose $(\del A_{\oplus 1})^1 = \del B(x_1; 1) \cap \del A_{\oplus 1} \neq \phi$. Clearly, we have $\lambda((\del A_{\oplus 1})^1) = \Omega ( (\del A_{\oplus 1})^1; x_1)$; i.e., the perimeter of the arc $(\del A_{\oplus 1})^1$ is simply the central angle of the arc $(\del A_{\oplus 1})^1$. Now comes our key observation: The angle subtended by the arc to $(\del A_{\oplus 1})^1$ to $x_0$, which is denoted by $\Omega((\del A_{\oplus 1})^1; x_0)$, is at least as large as $\Omega( (\del A_{\oplus 1})^1 ; x_1)/2$; i.e.,
\begin{align*}
\Omega((\del A_{\oplus 1})^1; x_0) \geq \frac{\Omega( (\del A_{\oplus 1})^1 ; x_1)}{2}.
\end{align*}
If $x_0$ lay on the circumference of $B(x_1; 1)$, this would be an exact equality by the inscribed angle theorem from geometry. In this case, the point $x_0$ might lie in the interior of $B(x_1,1)$, but it may be easily verified that the the angle subtended by the arc at $x_0$ would be at least as large as the inscribed angle of the arc. 
\end{itemize}

\begin{figure}
\begin{center}
\begin{tabular}{cc}
\includegraphics[width=3in]{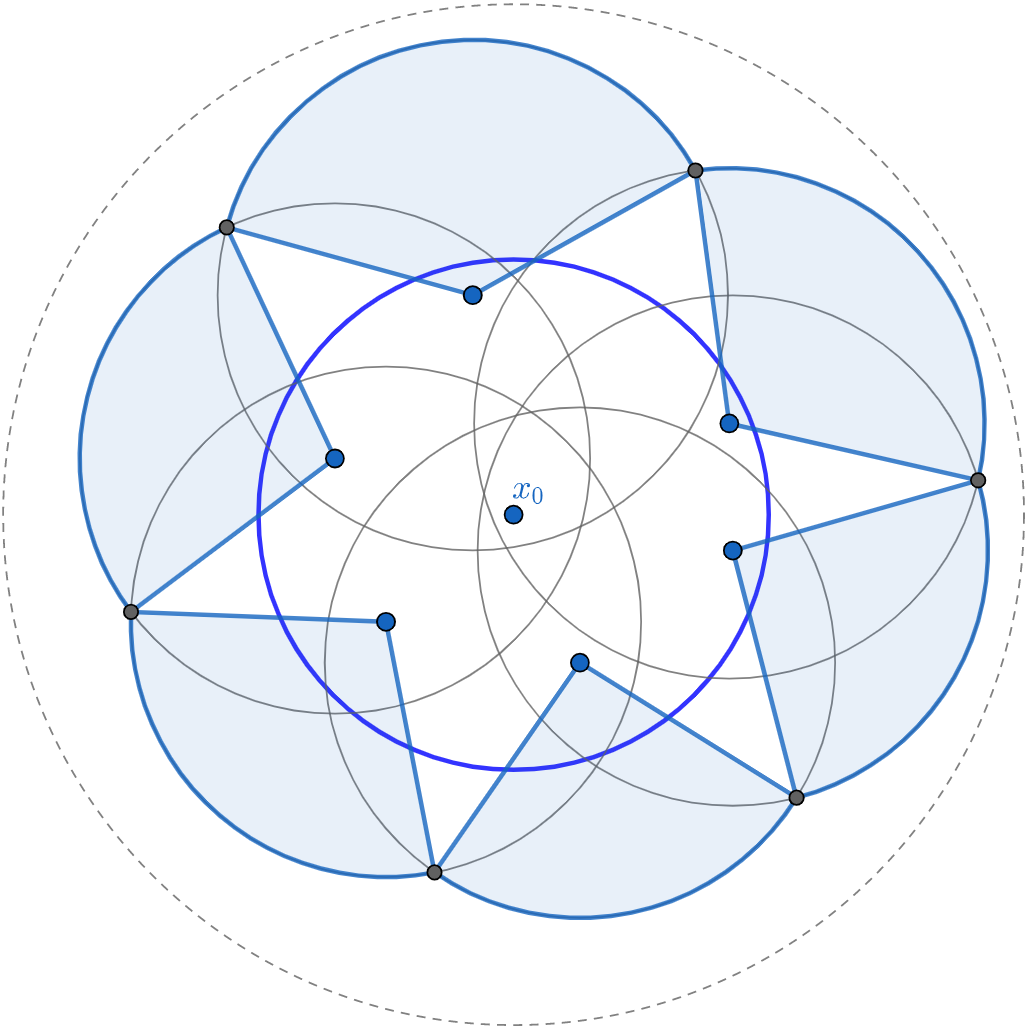} & \includegraphics[width=3in]{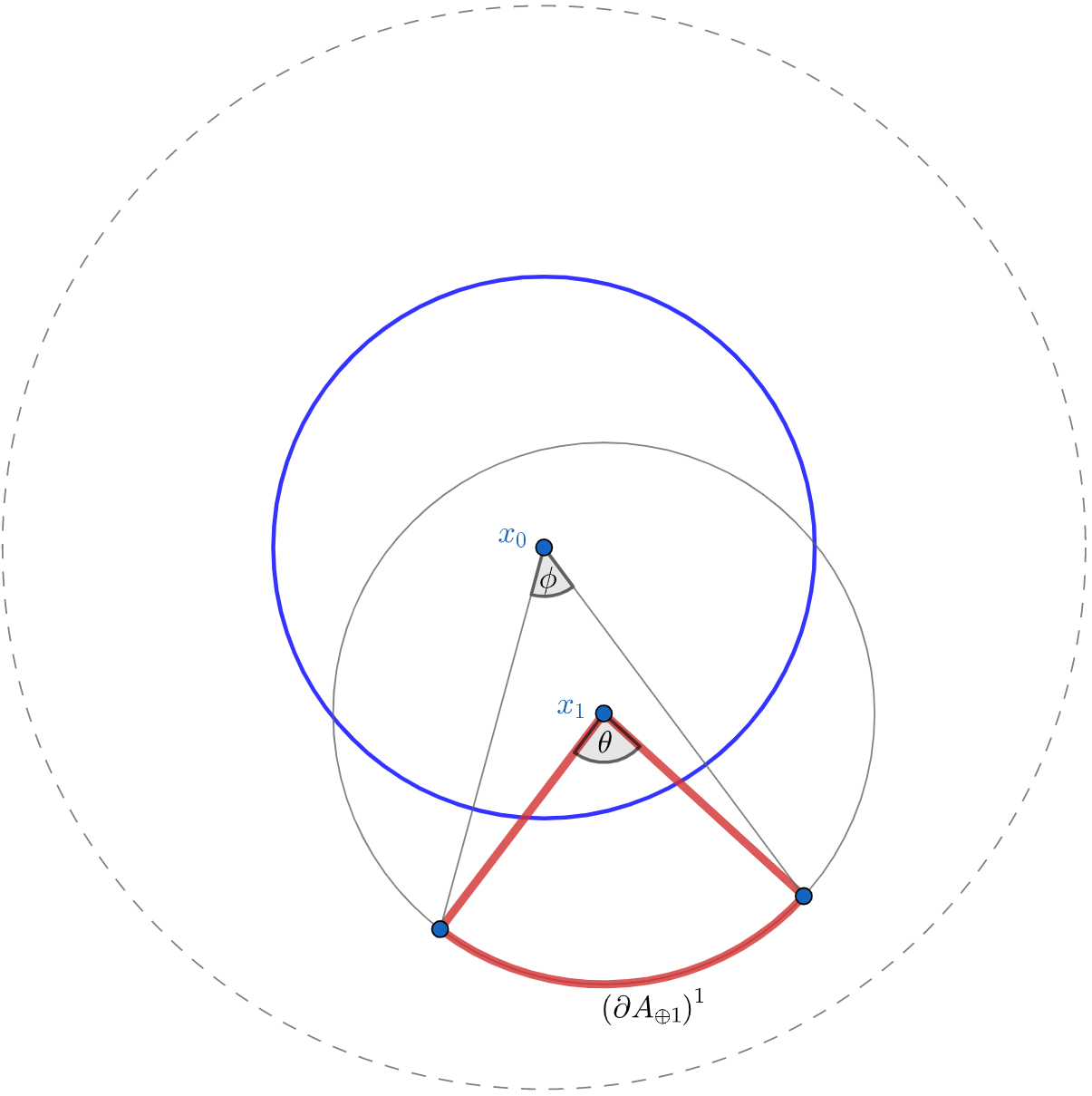} \\
(a) Boundary of $A_{\oplus 1}$ & (b) Angle subtended by $(\del A_{\oplus 1})^1$ at $x_0$  
\end{tabular}
\caption{In (a), the boundary of $A_{\oplus 1}$ is shown to be a union of arcs of various circles. In (b), it follows from the inscribed angle theorem that $\phi \geq \theta/2$.} \label{fig: puzzle_solution}
\end{center}
\end{figure}

Observations (1) and (2) are illustrated in Figure~\ref{fig: puzzle_solution}. We now combine observations (1) and (2). Since $A_{\oplus 1}$ is star-shaped from the point of view of $x_0$, we have
\begin{align*}
2\pi = \sum_{i=0}^N \Omega( (\del A_{\oplus 1})^i ; x_0).
\end{align*}
Using the inequality from observation (2), we obtain
\begin{align*}
2\pi = \sum_{i=0}^N \Omega( (\del A_{\oplus 1})^i ; x_0) \geq \frac{1}{2}\sum_{i=0}^N \Omega( (\del A_{\oplus 1})^i ; x_i) = \frac{\lambda(\del A_{\oplus 1})}{2}.
\end{align*}
This leads to $2\pi \geq \frac{\lambda(\del A_{\oplus 1})}{2}$, which completes the solution to the puzzle. \hfill $\square$

\paragraph{C-Puzzle:} Consider $N \geq 1$ points $x_i \in C(x_0; 1)$, for $1 \leq i \leq N$. Let $A = \{x_i \in \real^2 ~\mid~ 0 \leq i \leq N\}$. Show that the perimeter of $A_{\bp 1} = A \oplus C(1)$ is no more than that of $C(x_0; 2)$; i.e.,
\begin{align}\label{eq: puzzle}
\lambda(\del A_{\bp 1}) \leq \lambda(\del C(2)) = 16.
\end{align}
Figure~\ref{fig: puzzlec} shows an example of the set $A_{\bp 1}$.
\begin{figure}
\begin{center}
\includegraphics[width = 3 in]{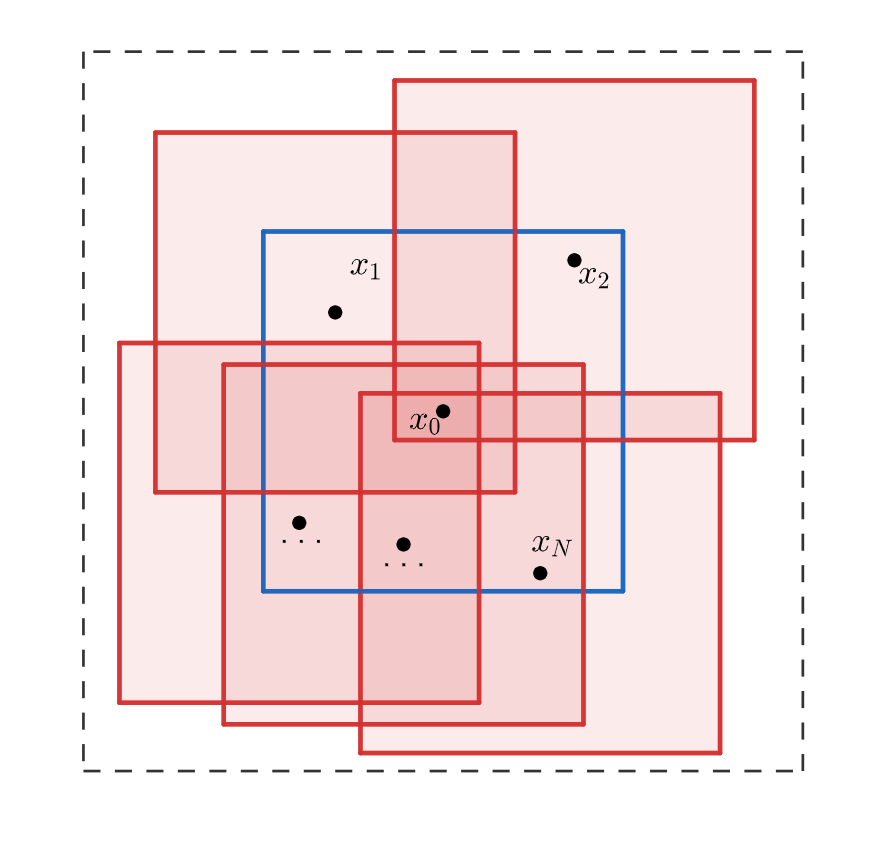}
\end{center}
\caption{Given $\{x_1, \dots, x_n\} \subseteq C(x_0; 1)$, the shaded region is the union of unit $\ell_\infty$-balls with centers at $x_i$, for $0 \leq i \leq N$. The problem is to upper-bound the perimeter of the shaded region with that of $C(x_0; 2)$.} \label{fig: puzzlec}
\end{figure}

\paragraph{Solution:}
The boundary $\del A_{\bp 1}$ consists of horizontal and vertical segments, and so the perimeter can be calculated by measuring the total length of the horizontal segments and the vertical segments. Observe that although the set $A_{\bp 1}$ is nonconvex, every axis-aligned line intersects the $\del A_{\bp 1}$ in at most two points. Lines that do not intersect $C(x_0; 2)$ also do not intersect $A_{\bp 1}$. For lines that do intersect $C(x_0; 2)$, we argue as follows. Consider the horizontal line $y = y_0$ that intersects the boundary $\del A_{\bp 1}$ in the points $(a_1, y_0), (a_2,y_0), \dots, (a_\ell,y_0)$ where $a_1 < \dots < a_\ell$. The points $(a, y_0)$ where $a \in (a_1, a_1+2)$ lie in the interior of $A_{\bp 1}$ and cannot lie on the boundary, and thus $a_2-a_1\geq 2$. Arguing similarly, we have $a_\ell - a_{\ell-1} \geq 2$. If $\ell \geq 4$, then $a_{\ell-1} > a_2$. Combining these inequalities we arrive at $a_\ell - a_1 > 4$, which is not possible since $A_{\bp 1} \subseteq C(x_0; 2)$. An identical argument also works for vertical lines. Thus, the projection of $\del A_{\bp 1}$ on the vertical axis is exactly twice the sum of all the vertical segments in $\del A_{\bp 1}$, which is bounded above by $8$. The same holds true for the horizontal segments, and we conclude that the perimeter of $A_{\bp 1}$ is bounded above by 16. \hfill $\square$


\section{Reverse Euclidean isoperimetric inequality for parallel sets}\label{sec: leb}

We first prove versions of the puzzles in Section~\ref{sec: puzzle} in $d$ dimensions.
\begin{proposition}\label{prop: puzzle_d_dim}
Let $r, \delta > 0$. Consider $N \geq 1$ points $x_i \in B_d(x_0; r)$, for $1 \leq i \leq N$. Let $A = \{x_i \in \real^d ~\mid~ 0 \leq i \leq N\}$. The surface area of $A_{\bc r}$ satisfies the following inequality: 
\begin{align}
\lambda(\del A_{\bc r}) \leq 2^{d-1} \Omega_d r^{d-1}.
\end{align}
\end{proposition}
\begin{proof}
Let the Voronoi region $D_i$ associated to each $x_i$ be defined as
$$D_i = \{x \in \real^d ~\mid~ \|x - x_j\|_2 \geq \|x - x_i\|_2 \text{ if } j \leq i \text{ and } \|x - x_j\|_2 > \| x - x_i \|_2 \text{ if } j > i\}.$$ Note that the $D_i$'s are pairwise disjoint convex regions, not necessarily bounded, which cover all of $\real^d$. Thus, we may write
\begin{align}
\lambda(\del A_{\bc r}) &= \sum_{i=0}^N \lambda(\del A_{\bc r} \cap D_i )\\
&\stackrel{(a)}= \sum_{i=0}^N \lambda(\del B(x_i; r) \cap D_i )\\
&= \sum_{i=0}^N r^{d-1} \Omega( \del B(x_i; r) \cap D_i ; x_i). \label{eq: thumb}
\end{align}
Here, step $(a)$ follows from the definition of the Voronoi region. 

Unfortunately, for $d > 2$, the inscribed angle theorem no longer holds; i.e., the solid angle subtended by a region on the sphere to an arbitrary point on the sphere is not a fixed fraction of the solid angle subtended to the center of the sphere. However, we are able to prove a lower bound on the inscribed angle in Lemma~\ref{lemma: inscribed_angle_d}, stated as follows.

\begin{lemma}\label{lemma: inscribed_angle_d}
The solid angle subtended by $\del B(x_i; r) \cap D_i$ at $x_0$ satisfies the bound
\begin{align*}
\Omega(\del B(x_i; r) \cap D_i; x_0) \geq \frac{\Omega(\del B(x_i; r) \cap D_i; x_i)}{2^{d-1}}.
\end{align*}
\end{lemma}
\begin{proof}
Let $S \defn \del B(x_i; r) \cap D_i \subseteq \del B(x_i; r)$, and assume that $S \neq \phi$ without loss of generality. Note that $x_0 \notin S$, but $x_0 \in B(x_i; r)$. Consider a small surface area element $dS$ in $S$ around a point $x \in S$. Note that $\Omega(dS; x_i) = \lambda(dS)/(r^{d-1}\Omega_d)$. As shown in Figure~\ref{fig: inscribed_angle_d_dim}, let $\angle x_i x x_0 = \theta$. Extend the line joining $x$ and $x_0$ to $\tilde x_0$ on $B(x_i; r)$. Using trigonometry, we can check that $d(x, \tilde x_0) = 2r\cos \theta$. Also, it is not hard to check that 
\begin{align}
\Omega(dS; x_0) &= \frac{\lambda(dS) \cos \theta}{\Omega_d d(x, x_0)^{d-1}}\\
&\geq \frac{\lambda(dS) \cos \theta}{\Omega_d d(x, \tilde x_0)^{d-1}}\\
&= \frac{\lambda(dS) \cos \theta}{\Omega_d (2r\cos\theta)^{d-1}}\\
&= \Omega(dS; x_i) \cdot \frac{1}{2^{d-1}\cos^{d-2}\theta}\\
&\geq \frac{\Omega(dS; x_i)}{2^{d-1}}. \label{eq: toby}
\end{align}
Noting that $\Omega(S; x_0) = \int_S \Omega(dS; x_0)$ and $\Omega(S; x_i) = \int_S \Omega(dS; x_i)$, we may integrate the inequality in \eqref{eq: toby} to conclude the desired result.
\begin{figure}
\begin{center}
\includegraphics[width = 3 in]{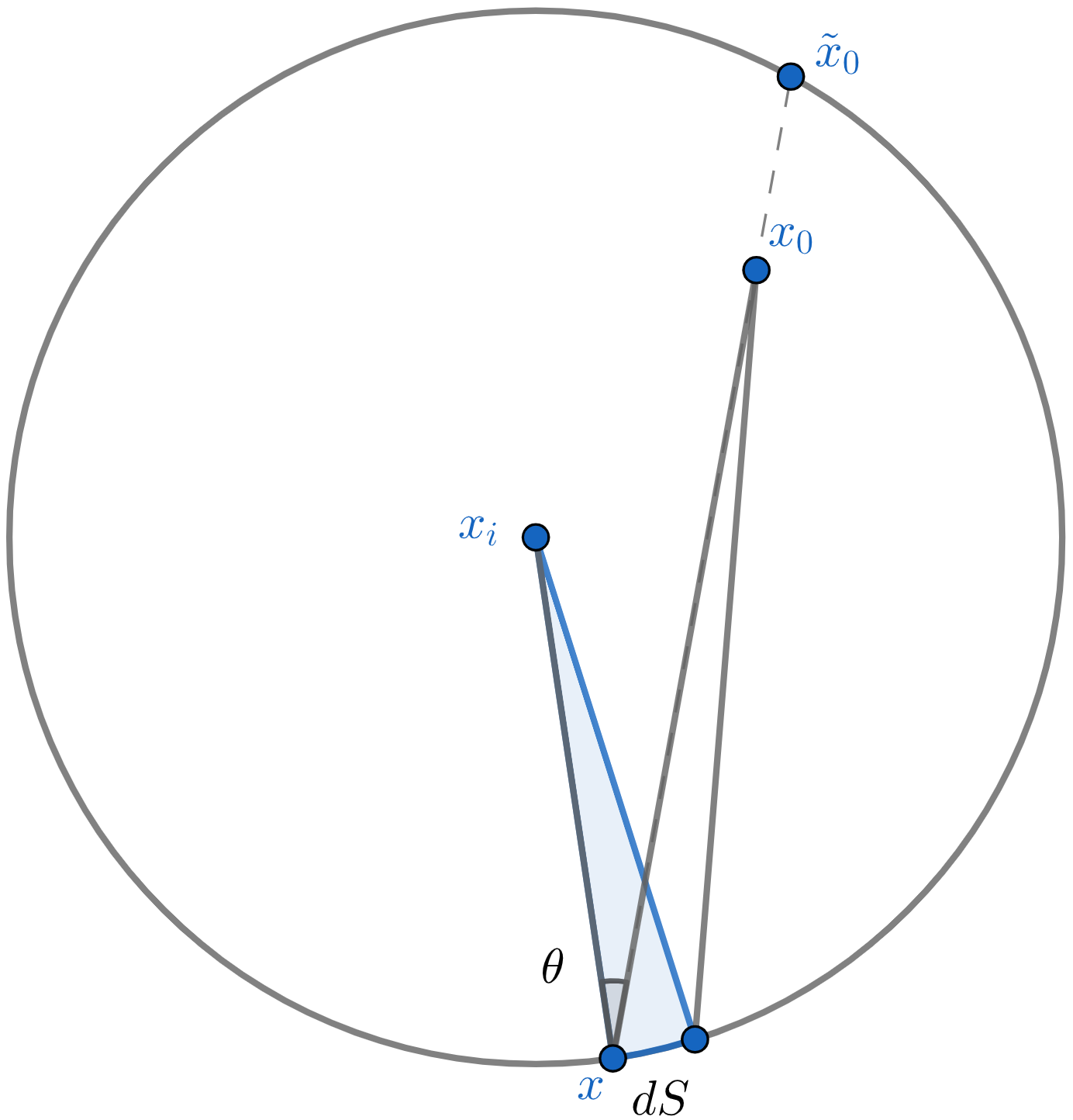}
\caption{Comparison between the angle subtended by a small area element $dS$ to $x_0$ and $x_i$.} 
\label{fig: inscribed_angle_d_dim}
\end{center}
\end{figure}
\end{proof}

\begin{lemma}\label{lemma: A_r_star_shaped}
The set $A_{\bc r}$ is star-shaped from the point of view of $x_0$.
\end{lemma}
\begin{proof}
The proof is essentially identical to observation (1) from Section~\ref{sec: puzzle}, so we omit it.
\end{proof}
We can now complete the proof of Proposition~\ref{prop: puzzle_d_dim} as follows: 
\begin{align*}
\lambda(\del A_{\bc r}) &\stackrel{(a)}= \sum_{i=0}^N r^{d-1} \Omega( \del B(x_i; r) \cap D_i ; x_i)\\
&\stackrel{(b)}\leq 2^{d-1} \sum_{i=0}^N r^{d-1} \Omega(\del B(x_i; r) \cap D_i; x_0)\\
&\stackrel{(c)}= 2^{d-1} \Omega_d r^{d-1}.
\end{align*}
Here, step $(a)$ is the equality from equation~\eqref{eq: thumb}, step $(b)$ follows from Lemma~\ref{lemma: inscribed_angle_d} and step $(c)$ follows from Lemma~\ref{lemma: A_r_star_shaped}.
\end{proof}

\begin{proposition}\label{prop: boxes}
Let $r, \delta > 0$. Consider $N \geq 1$ points $x_i \in C_d(x_0; r)$, for $1 \leq i \leq N$. Let $A = \{x_i \in \real^d ~\mid~ 0 \leq i \leq N\}$. The surface area of $A_{\bp r}$ satisfies the following inequality: 
\begin{align}
\lambda(\del A_{\bp r}) \leq 2^{2d-1} d r^{d-1}.
\end{align}
\end{proposition}
\begin{proof}
The following two observations are crucial: (1) The normal vector to the boundary of $\del A_{\bp r}$ exists almost everywhere and is aligned with one of the coordinate axes; and (2) Any axis-aligned line that intersects surface $\del A_{\bp r}$ at most twice. 

Observation (1) is immediate since $A_{\bp r}$ is a union of finitely many cubes. Without loss of generality, consider the axis-parallel line $y(t) = t(1, 0, \dots, 0) + (0, v_1, \dots, v_{d-1})$ for some $v \defn (v_1, \dots, v_{d-1}) \in \real^{d-1}$. Suppose that this line intersects $\del A_{\bp r}$ in the points $p_i = (t_i, v)$ for $1 \leq i \leq \ell$ such that $t_1 < t_2 \dots < t_\ell$. Note that $\ell$ must be an even number, so $t_2 \neq t_{\ell-1}$. The point $(v, t_\ell)$, which is the rightmost point on the line $y(t)$, lies on the boundary of some $C_d(x_i; r)$. Thus, all the points $\{(t, v) ~|~ t \in (t_\ell-2r, v)$ lie in the interior of $C_d(x_i; r)$ and cannot lie on the boundary $\del A_{\bp r}$. This gives the inequality $t_{\ell-1} \leq t_\ell - 2r$. A similar argument for $t_1$, which is the leftmost point on the line $y(t)$, gives the inequality $t_2 \geq t_1 + 2r$. If $\ell \geq 4$, then $t_\ell - 2r \geq t_{\ell-1} > t_2 \geq t_1 + 2r$ , giving $t_\ell - t_1 > 4r$. Since $A_{\bp r} \subseteq C_d(x_0; 2r)$, we must have $t_\ell - t_1 \leq 4r$, which is a contradiction. This means the assumption $\ell \geq 3$ was incorrect and $\ell$ is at most 2.

Observation (1) gives that the surface area measured via $\lambda_C$ is identical to that measured using $\lambda_B$. Moreover, this value is given by the surface integral
\begin{align*}
\lambda(\del A_{\bp r}) = \sum_{i=1}^d \int_{\del A_{\bp r}} \abs{n_i \cdot s} ds,
\end{align*}
where $n_i$ for $i \in [d]$ are the standard basis vectors. Observation (2) gives that for each $i \in [d]$,
\begin{align*}
\int_{\del A_{\bp r}} \abs{n_i \cdot s} ds &\leq \int_{\del C_d(x_0; 2r)} \abs{n_i \cdot s}
= 2(4r)^{d-1}.
\end{align*}
This gives the surface area inequality 
\begin{align}\label{eq: 2dr}
\lambda(\del A_{\bp r}) \leq 2d (4r)^{d-1}.
\end{align}
\end{proof}
Our next proposition establishes an inequality for the volumes $\lambda(A_{\bc r+\delta} \setminus A_{\bc r})$ and $\lambda(A_{\bp r+\delta} \setminus A_{\bp r})$ using Propositions~\ref{prop: puzzle_d_dim} and ~\ref{prop: boxes}, respectively. 
\begin{proposition}\label{prop: kneser}
Let $r, \delta > 0$. Consider $N \geq 1$ points $x_i \in B_d(x_0; r)$, for $1 \leq i \leq N$. Let $A = \{x_i \in \real^d ~\mid~ 0 \leq i \leq N\}$. Then the following inequality holds:
\begin{align}\label{eq: ballz}
\lambda(A_{\bc r+\delta} \setminus A_{\bc r}) &\leq 2^{2d-1} ((r+\delta)^d - r^d).
\end{align}
Similarly, if $A = \{x_i \in \real^d ~\mid~ 0 \leq i \leq N\}$ where $x_i \in C_d(x_0; r)$, then the following inequality holds:
\begin{align}\label{eq: boxez}
\lambda(A_{\bp r+\delta} \setminus A_{\bp r}) &\leq 2^{2d-1} ((r+\delta)^d - r^d).
\end{align}
\end{proposition}
\begin{proof}
We use the following result from Stacho~\cite{Sta76}, which is a generalization of Kneser's Lemma~\cite{Kne51}:

\begin{lemma}[Theorem 4 from Stacho~\cite{Sta76}] \label{lemma: kneser}
Let $K$ be a bounded, centrally-symmetric, convex set in $\real^d$ and let $A$ be an arbitrary bounded set in $\real^d$. Then for any $0< a \leq b$ and any $t \geq 1$,
\begin{align*}
\lambda( (A\oplus tb K) \setminus (A\oplus ta K)) \leq t^d \lambda( (A\oplus b K) \setminus (A\oplus a K)).
\end{align*}
\end{lemma}

We shall now apply Lemma~\ref{lemma: kneser} to upper bound $\lambda(A_{\bc r+\delta} \setminus A_{\bc r})$. Let $M \in \mathbb N$ and set $t \defn (1 + \delta/r)^{1/M}$. By Lemma~\ref{lemma: kneser}, we have that
\begin{align*}
\lambda(A_{\bc r+\delta} \setminus A_{\bc r}) &= \lambda(A_{\bc r t^M} \setminus A_{\bc r})\\
&= \sum_{i=1}^M \lambda(A_{\bc rt^i} \setminus A_{\bc rt^{i-1}})\\
&\leq \left( \sum_{i=1}^M t^{(i-1)d}\right) \lambda(A_{\bc rt} \setminus A_{\bc r}) \\
&\stackrel{(a)}= \frac{t^{Md}-1}{t^d-1} \cdot \lambda(A_{\bc rt} \setminus A_{\bc r})\\
&= \frac{t^{Md}-1}{t^d-1} \frac{ \lambda(A_{\bc rt} \setminus A_{\bc r})}{rt-r} \cdot (rt-r).
\end{align*}
Step $(a)$ follows from Lemma~\ref{lemma: kneser}. Taking the limit as $M \to \infty$ and $t \to 1_+$,
\begin{align}\label{eq: boxplus1}
\lim_{t \to 1_+} \frac{\lambda(A_{\bc rt} \setminus A_{\bc r})}{rt-r} \stackrel{(a)}= \lambda(\del A_{\bc r}) \stackrel{(b)}\leq 2^{d-1} \Omega_d r^{d-1}. 
\end{align}
Here, the existence of the limit in step $(a)$ follows from Stacho~\cite{Sta76}, and the inequality in $(b)$ follows from Proposition~\ref{prop: puzzle_d_dim}.
Additionally, we have the limits
\begin{align}
\lim_{M \to \infty} \frac{t^{Md}-1}{t^d-1} \cdot (rt-r) &= \lim_{M \to \infty} \frac{(1+\delta/r)^d-1}{(1+\delta/r)^{d/M} - 1} \cdot r ((1+\delta/r)^{1/M} - 1) \nonumber \\
&= ((1+\delta/r)^d-1)\frac{r}{d}\nonumber \\
&= \frac{(r+\delta)^d - r^d}{dr^{d-1}}. \label{eq: boxplus2}
\end{align}
Combining inequalities~\eqref{eq: boxplus1} and ~\eqref{eq: boxplus2} and noting that $\Omega_d/d = \omega_d$, we arrive at
\begin{align}
\lambda(A_{\bc r+\delta} \setminus A_{\bc r}) \leq 2^{d-1} \omega_d ((r+\delta)^d - r^d). 
\end{align}
Equation~\eqref{eq: boxez} is proved similarly. The only difference is in inequality~\eqref{eq: boxplus1}, which changes to
\begin{align}
\lim_{t \to 1_+} \frac{\lambda(A_{\bp rt} \setminus A_{\bp r})}{rt-r} = \lambda_C(\del A_{\bp r}) \stackrel{(a)}\leq (2d)(4r)^{d-1},
\end{align}
where the inequality in step $(a)$ follows from Proposition~\ref{prop: boxes}. Combining this with equation~\eqref{eq: boxplus2}, we conclude
\begin{align}
\lambda(A_{\bp r+\delta} \setminus A_{\bp r}) \leq 2^{2d-1} ((r+\delta)^d - r^d).
\end{align}
\end{proof}

Before stating our next proposition, we define the packing number of a set in $\real^d$.
\begin{definition}
Let $A \subseteq \real^d$ be a measurable set and let $\epsilon > 0$. A collection of points denoted by $\text{Packing}(A; \epsilon) \defn \{x_i ~\mid~ 1 \leq i \leq N\}$ is said to be an $\epsilon$-packing of $A$ if for every $x, y \in \text{Packing}(A; \epsilon)$, we have $d(x,y) > \epsilon$, where $d(\cdot, \cdot)$ is a metric on $\real^d$. The $\epsilon$-packing number of $A$, denoted by $N_B(A; \epsilon)$ and $N_C(A; \epsilon)$ for the $\ell_2$ and $\ell_\infty$ metrics, respectively, is the largest size of an $\epsilon$-packing of $A$.
\end{definition}

\begin{proposition}\label{prop: packing}
Let $r, \delta > 0$. Consider $N \geq 1$ arbitrary points $x_1, \dots, x_N$ in $\real^d$, and let $A = \{x_1, \dots, x_N\}$. Then 
\begin{align}\label{eq: prop2}
\lambda(A_{\bc r+\delta} \setminus A_{\bc r}) \leq N_B(A; r) \cdot 2^{d-1} \omega_d ((r+\delta)^d - r^d),
\end{align}
and
\begin{align}\label{eq: prop2_box}
\lambda(A_{\bp r+\delta} \setminus A_{\bp r}) \leq N_C(A; r) \cdot 2^{2d-1}  ((r+\delta)^d - r^d),
\end{align}
\end{proposition}
\begin{corollary}\label{cor: prop2}
The surface area of $A_{\bc r}$ satisfies 
\begin{align*}
\lambda(\del A_{\bc r}) \leq N_B(A; r)2^{d-1}  \omega_d d r^{d-1}.
\end{align*}
Similarly, the surface area of $A_{\bp r}$ satisfies 
\begin{align*}
\lambdacsup(\del A_{\bp r}) \leq N_C(A; r)2^{2d-1}d r^{d-1}.
\end{align*}
\end{corollary}

\begin{proof}
The proof of Corollary~\ref{cor: prop2} follows by taking the limit as $\delta \to 0$, so we shall only prove the bound~\eqref{eq: prop2}. We shall first prove the result for parallel sets with respect to $B$. In what follows, let $d(\cdot, \cdot)$ be the $\ell_2$ distance on $\real^d$. 

Let $\widehat A = \{\widehat x_1, \dots, \widehat x_k\} \subseteq A$ be a maximal $r$-packing of the set $A$; i.e., if $i \neq j$, then $d(\widehat x_i, \widehat x_j) > r$, but $d(x, \widehat A) \leq r$ for every $x \in A$. Note that $k = N_B(A; r)$. Let $\widehat A^i = \{x \in A ~\mid~ d(x, \widehat x_i) \leq r\}$. Observe that $\cup_{i=1}^k \widehat A^i = A$, but the $\widehat A^i$'s need not be mutually exclusive. This means that 
\begin{align*}
A_{\bc r+\delta} \setminus A_{\bc r} &= \cup_{i=1}^k  \left( (\widehat A^i)_{\bc r+\delta}\right) \setminus A_{\bc r}\\
&= \cup_{i=1}^k  \left( (\widehat A^i)_{\bc r+\delta} \setminus A_{\bc r} \right)\\
&\subseteq \cup_{i=1}^k  \left( (\widehat A^i)_{\bc r+\delta} \setminus (\widehat A^i)_{\bc r} \right), 
\end{align*}
so
\begin{align}
\lambda(A_{\bc r+\delta} \setminus A_{\bc r}) &\leq \lambda\left( \cup_{i=1}^k  \left( (\widehat A^i)_{\bc r+\delta} \setminus (\widehat A^i)_{\bc r} \right)\right)\\
&\leq \sum_{i=1}^k \lambda \left( (\widehat A^i)_{\bc r+\delta} \setminus (\widehat A^i)_{\bc r} \right) \label{eq: prop2_sum}\\
&\stackrel{(a)} \leq k \cdot 2^{d-1} \omega_d ((r+\delta)^d - r^d)\\
&\stackrel{(b)}\leq N_B(A; r) \cdot 2^{d-1} \omega_d ((r+\delta)^d - r^d).
\end{align}
Here, inequality $(a)$ follows directly from Proposition~\ref{prop: puzzle_d_dim}, and $(b)$ follows from the maximal property of the packing. The proof for parallel sets with respect to $C$ is identical to the one for $B$, so we shall omit it.
\end{proof}

\begin{theorem}\label{thm: lebesgue_rip_volume}
Let $r, \delta > 0$, and let $A_{\bc r}$ be an $r$-parallel set in $\real^d$ satisfying $\lambda(A_{\bc r}) \leq V$. Then the following inequalities hold:
\begin{align*}
\lambda(A_{\bc r+\delta} \setminus A_{\bc r}) &\leq \frac{V}{r^d} \cdot \left(2^{2d-1} ((r+\delta)^d - r^d)\right), \quad \text{ and }\\
\lambda(\del A_{\bc r}) &\leq \frac{V}{ r} \cdot \left(2^{2d-1} d \right).
\end{align*}
Similarly, if $\lambda(A_{\bc r}) \leq V$, then the following inequalities hold:
\begin{align*}
\lambda(A_{\bp r+\delta} \setminus A_{\bp r}) &\leq \frac{V}{r^d} \cdot \left(2^{2d-1} ((r+\delta)^d - r^d)\right), \quad \text{ and }\\
\lambda_C(\del A_{\bp r}) &\leq \frac{V}{ r} \cdot \left(2^{2d-1} d \right).
\end{align*}
\end{theorem}

\begin{proof}
We prove the result for parallel sets of the form $A_{\bc r}$ first. Since $A$ is compact (closed and bounded), for each $n \geq 1$, there exists a finite set $A^n \subseteq A$ such that $A \subseteq (A^n)_{1/n}$. Clearly, the sets $A^n$ converge to $A$ in the Hausdorff metric, since $d_{\text{Hausdorff}(A, A^n)} \leq 1/n$. For any $t > 0$, we have $(A^n)_{\bc t} \subseteq A_{\bc t}$ and $A_{\bc t} \subseteq (A^n)_{\bc t +1/n}$. Equivalently, for all large enough $n$, we have the inclusion
\begin{align*}
A_{\bc t-1/n} \subseteq (A^n)_{\bc t} \subseteq A_{\bc t},
\end{align*}
which implies
\begin{align*}
\lambda(A_{\bc t-1/n}) \leq \lambda((A^n)_{\bc t}) \leq \lambda(A_{\bc t}).
\end{align*}
From Stacho~\cite{Sta76}, the volume function $t \to \lambda(A_{\bc t})$ is continuous, so 
\begin{align}\label{eq: stacho_volume}
\lambda((A^n)_{\bc t}) \to \lambda(A_{\bc t}). 
\end{align}
Note that
\begin{align*}
\lambda((A^n)_{\bc r+\delta} \setminus (A^n)_{\bc r}) = \lambda((A^n)_{\bc r+\delta}) - \lambda((A^n)_{\bc r})) &\stackrel{(a)} \leq N_B(A^n; r) \cdot 2^{d-1} \omega_d ((r+\delta)^d - r^d)\\
&\stackrel{(b)}\leq N_B(A; r) \cdot 2^{d-1} \omega_d ((r+\delta)^d - r^d),
\end{align*}
where $(a)$ is true because each of the sets $A^n$ consists of finitely many points and Proposition~\ref{prop: packing} may be applied; and $(b)$ is true since $A^n \subseteq A$. Taking the limit as $n\to \infty$ and using the volume convergence from equation~\eqref{eq: stacho_volume}, we conclude that
\begin{align}\label{eq: A_thick}
\lambda(A_{\bc r+\delta} \setminus A_{\bc r}) \leq N_B(A; r) \cdot 2^{d-1} \omega_d ((r+\delta)^d - r^d).
\end{align}
The last step is to bound $N_B(A; r)$ in terms of the volume of $A_{\bc r}$. Consider any $r$-packing of $A$ given by $\text{Packing}(A; r) = \{x_i ~\mid~ 1 \leq i \leq N\}$. Clearly, we have $\cup_{i=1}^N B(x_i; r/2) \subseteq A_{\bc r/2} \subseteq A_{\bc r}$. Comparing volumes and noting that $B(x_i; r/2)$ are disjoint, we conclude that
\begin{align*}
N_B(A; r) \leq \frac{V}{\lambda(B(r/2))} = \frac{V}{\omega_d (r/2)^d},
\end{align*}
leading to the inequality 
\begin{align*}
\lambda(A_{\bc r+\delta} \setminus A_{\bc r}) \leq \frac{V}{r^d} \cdot 2^{2d-1}  ((r+\delta)^d - r^d).
\end{align*}
To prove the bound on $\del(A_{\bc r})$, we can divide both sides in inequality~\eqref{eq: A_thick} by $\delta$ and take the limit as $\delta \to 0$. By the results in Stacho~\cite{Sta76}, the limit of the left hand side exists, and we conclude that
\begin{equation*}
\lambda(\del A_{\bc r}) \leq \frac{V}{ r} \cdot \left(2^{2d-1} d \right).
\end{equation*}

The proof for $A_{\bp r}$ follows the same lines as above. The only change is the continuity of the volume function $\lambda(A_{\bp t})$ (as opposed to $\lambda(A_{\oplus t})$), which is also provided by Stacho~\cite{Sta76}. Letting $N_C(A; r)$ denote the packing number of $A$ with respect to $\ell_\infty$-balls of radius $r$, we arrive at the bound
\begin{align}\label{eq: A_thick_box}
\lambda(A_{\bp r+\delta} \setminus A_{\bp r}) \leq N_C(A; r) \cdot 2^{2d-1} ((r+\delta)^d - r^d).
\end{align}
Bounding $N_C(A; r)$ by $V/r^d$, we conclude 
\begin{align}\label{eq: A_thick_box_pack}
\lambda(A_{\bp r+\delta} \setminus A_{\bp r}) \leq \frac{V}{r^d} \cdot 2^{2d-1} ((r+\delta)^d - r^d).
\end{align}
Dividing by $\delta$ and taking the limit as $\delta \to 0$, 
\begin{equation*}
\lambda_C(\del A_{\bp r}) \leq \frac{V}{r} \cdot (d2^{2d-1}).
\end{equation*}	
\end{proof}

\begin{remark}\label{rem: kostya}
For convex sets of the form $A_{\bc r}$, the results in Chernov et al.~\cite{CheEtal19} yield a tighter bound than our Theorem~\ref{thm: lebesgue_rip_volume}. If $\lambda(A_{\bc r}) \leq V$, then
\begin{align*}
\lambda(\del A_{\bc r}) \leq \frac{V}{r}.
\end{align*}
\end{remark}

\begin{corollary}\label{cor: lebesgue_rip}
Let $R \geq r >0$, and let $\delta > 0$. Let $A$ be an arbitrary closed set contained in $B(R)$. Then the following bounds hold for the $r$-parallel set $A_{\bc r}$:
\begin{align*}
\lambda(A_{\bc r+\delta} \setminus A_{\bc r}) &\leq \frac{(R+r/2)^d}{(r/2)^d} \cdot 2^{d-1} \omega_d ((r+\delta)^d - r^d), \quad \text{ and }\\
\lambda(\del A_{\bc r}) &\leq \frac{(R+r/2)^d}{(r/2)^d} \cdot \left(2^{d-1} d \omega_d r^{d-1}\right).
\end{align*}
The corresponding bound for $A_{\bp r}$ is as follows:
\begin{align*}
\lambda(A_{\bp r+\delta} \setminus A_{\bp r}) &\leq \frac{\omega_d(R+ r\sqrt d/2)^d}{r^d} \cdot 2^{2d-1} ((r+\delta)^d - r^d), \quad \text{ and }\\
\lambda_C(\del A_{\bp r}) &\leq \frac{\omega_d(R+ r\sqrt d/2)^d d}{r} \cdot 2^{2d-1} .
\end{align*}

\end{corollary}
\begin{proof}
The proof is identical to that of Theorem~\ref{thm: lebesgue_rip_volume}, with the only change being that the $r$-packing number of $A$ is bounded as
\begin{align*}
N_B(A; r) \leq \frac{(R+r/2)^d}{(r/2)^d}.
\end{align*}
The packing number $N_C(A; r)$ is bounded using the crude upper bound $A \oplus (r/2)C \subseteq B(R + r\sqrt d/2)$ to obtain
\begin{align*}
N_C(A; r) \leq \frac{\omega_d(R+ r\sqrt d/2)^d }{r^d}.
\end{align*}
\end{proof}


\section{Reverse Gaussian isoperimetric inequality for parallel sets}\label{sec: gauss}

We now prove a reverse isoperimetry inequality for parallel sets under the Gaussian measure in $\real^d$. In the Gaussian isoperimetric inequality, the notion of Gaussian surface area used is the lower-Gaussian surface area. For reverse isoperimetric inequalities, it makes more sense to use the upper-Gaussian surface, since the aim is to provide upper bounds on the Gaussian surface area. For many well-behaved sets such as convex sets or sets with twice-continuously differentiable boundaries, the two notions of surface areas are identical~\cite{Naz03}. Moreover, the Gaussian surface area is obtained by integrating the Gaussian distribution with respect to the Hausdorff measure on the boundary of the set; i.e.,
\begin{align*}
\gamma(\del A) = \lim_{\delta \to 0} \frac{\gamma(A_\delta \setminus A)}{\delta} = \int_{\del A} \phi(x) d\cH^{d-1}(x),
\end{align*}
where $\cH^{d-1}$ is the $(d-1)$-dimensional Hausdorff measure. We do not investigate whether the Gaussian surface area of parallel sets is also given by such a surface integral. Our main result is as follows:
\begin{theorem}\label{thm: gaussian_rip}
Let $A \subseteq \real^d$ be an arbitrary closed set, and for $r>0$, consider the $r$-parallel set $A_{\bc r}$. Let $\delta > 0$. Then the upper-Gaussian surface area of $A_{\bc r}$ satisfies the bound
\begin{align*}
\bar \gamma(\del A_{\bc r}) \leq \max \left(C, \frac{C}{r} \right),
\end{align*}
where $C$ is a dimension-dependent constant that grows like $e^{\Theta(d)}$. Similarly, the the upper-Gaussian surface area of $A_{\bp r}$ satisfies the bound
\begin{align*}
\bar \gamma_C(\del A_{\bp r}) \leq \max \left(C, \frac{C}{r} \right),
\end{align*}
where $C$ is a dimension-dependent constant that grows like $e^{\Theta(d)}$.
\end{theorem}
\begin{proof}
We first prove the result for sets of the form $A_{\bc r}$. Our proof relies on an observation in Ball~\cite{Ball93} which leads to a (loose) upper bound of $\sqrt d$ on the Gaussian surface area of arbitrary convex sets in $\real^d$. (The tight upper-bound is $\Theta(d^{1/4})$.) The observation is simple:
\begin{align*}
e^{-\frac{\|x\|^2}{2}} = \int_{\tau = 0}^\infty \tau e^{-\tau^2/2} \mathbbm 1\{x \in B(\tau)\} d\tau.
\end{align*}

Using this, we rewrite $\gamma(A_{\bc r+\delta} \setminus A_{\bc r})$ as
\begin{align}
\gamma(A_{\bc r+\delta} \setminus A_{\bc r}) &= \frac{1}{(2\pi)^{d/2}}\int_{A_{\bc r+\delta} \setminus A_{\bc r}} e^{-\|x\|^2/2} dx\\
&= \frac{1}{(2\pi)^{d/2}}\int_{A_{\bc r + \delta} \setminus A_{\bc r}} \int_{\tau=0}^\infty \tau e^{-\tau^2/2} \mathbbm 1\{x \in B(\tau)\} d\tau dx\\
&= \frac{1}{(2\pi)^{d/2}}\int_{\tau=0}^\infty  \tau e^{-\tau^2/2} \left(\int_{A_{\bc r + \delta} \setminus A_{\bc r}}  \mathbbm 1\{x \in B(\tau)\} dx \right) d\tau\\
&= \frac{1}{(2\pi)^{d/2}}\int_{\tau=0}^\infty  \tau e^{-\tau^2/2} \lambda\left( (A_{\bc r + \delta} \setminus A_{\bc r}) \cap B(\tau)\right) d\tau\\
%
%
%
&\stackrel{(a)}= \frac{1}{(2\pi)^{d/2}}\int_{\tau=0}^\infty  \tau e^{-\tau^2/2} \lambda\left( ((A^\tau)_{\bc r+\delta} \setminus (A^\tau)_{\bc r}) \cap B(\tau) \right) d\tau\\
&\leq \frac{1}{(2\pi)^{d/2}}\int_{\tau=0}^\infty  \tau e^{-\tau^2/2} \lambda\left( ((A^\tau)_{\bc r+\delta} \setminus (A^\tau)_{\bc r})  \right) d\tau. \label{eq: A_tau}
\end{align}
In $(a)$, we let $A^\tau \defn A \cap B(\tau+r+\delta)$. Since points in $A \setminus A^\tau$ are more than $r+\delta$ distance away from $B(\tau)$, we may think of $A^\tau$ as the part of $A$ that is relevant to $B(\tau)$. In particular, the set $(A_{\bc r + \delta} \setminus A_{\bc r}) \cap B(\tau)$ is identical to $((A^\tau)_{\bc r + \delta} \setminus (A^\tau)_{\bc r}) \cap B(\tau)$. Now observe that $(A^\tau)$ is a closed set in $B(\tau+r+\delta)$. Using Corollary~\ref{cor: lebesgue_rip}, we may upper-bound $\lambda((A^\tau)_{\bc r + \delta} \setminus (A^\tau)_{\bc r})$ by
\begin{align*}
\lambda((A^\tau)_{\bc r + \delta} \setminus (A^\tau)_{\bc r}) \leq \frac{(\tau+\delta+3r/2)^d}{(r/2)^d} \cdot \left(2^{d-1} \Omega_d \frac{(r+\delta)^d - r^d}{d}\right).
\end{align*}
Substituting into inequality~\eqref{eq: A_tau}, we obtain
\begin{align*}
\gamma(A_{\bc r + \delta} \setminus A_{\bc r}) &\leq \frac{1}{(2\pi)^{d/2}} \left(2^{d-1} \Omega_d \frac{(r+\delta)^d - r^d}{d}\right) \cdot \frac{1}{(r/2)^d}\int_{\tau=0}^\infty   \tau e^{-\tau^2/2} (\tau+\delta+3r/2)^d  d\tau.
\end{align*}
We may expand $(\tau + \delta + 3r/2)^d = \sum_{i=0}^d {d \choose i} \tau^i (\delta + 3r/2)^{d-i}$. Using the closed-form expression
\begin{align*}
\int_{\tau = 0}^\infty e^{-\tau^2/2} \tau^{i+1} d\tau = 2^{i/2} \Gamma(1+ i/2),
\end{align*}
we arrive at
\begin{align*}
\gamma(A_{\bc r + \delta} \setminus A_{\bc r}) &\leq \frac{1}{(2\pi)^{d/2}} \left(2^{d-1} \Omega_d \frac{(r+\delta)^d - r^d}{d}\right) \cdot \frac{1}{(r/2)^d} \sum_{i=0}^d {d \choose i} 2^{i/2} \Gamma(1 + i/2) (3r/2 + \delta)^{d-i}.
\end{align*}
Dividing both sides by $\delta$ and taking the $\limsup$ as $\delta \to 0$,
\begin{align}
\bar \gamma(\del A_{\bc r}) &\leq \frac{1}{(2\pi)^{d/2}} \left(2^{d-1} \Omega_d r^{d-1}\right) \cdot \frac{1}{(r/2)^d} \sum_{i=0}^d {d \choose i} 2^{i/2} \Gamma(1 + i/2) (3r/2)^{d-i}\\
&= \frac{1}{(2\pi)^{d/2}} \left(2^{2d-1} \Omega_d \right) \cdot \sum_{i=0}^d C_i r^{i-1}, \label{eq: r_star}
\end{align}
where $C_i \defn {d \choose i} 2^{(d-i)/2} \Gamma(1 + (d-i)/2) (3/2)^{i}$. For simplicity, let us assume that $r \leq 1$. Note that for $r > 1$, the bound for $r=1$ continues to be valid. This is because any $r$-parallel set $A_{\bc r}$ for $r >1$ is also a $1$-parallel set for some set $\tilde A$. Denote the bound for $r = 1$ by $C$; i.e.,
\begin{align*}
C \defn \frac{1}{(2\pi)^{d/2}} \left(2^{2d-1} \Omega_d \right) \cdot \sum_{i=0}^d C_i.
\end{align*}
For $r \leq 1$, we conclude that
\begin{align*}
\bar \gamma(\del A_{\bc r}) &\leq \frac{1}{(2\pi)^{d/2}} \left(2^{2d-1} \Omega_d \right) \cdot  \sum_{i=0}^d C_i r^{i-1}\\
&\leq \frac{1}{(2\pi)^{d/2}} \left(2^{2d-1} \Omega_d \right) \cdot  \sum_{i=0}^d C_i \cdot \frac{1}{r}\\
&= \frac{C}{r}.
\end{align*}
This leads to our final bound,
\begin{align*}
\bar \gamma(\del A_{\bc r}) \leq \max \left(C, \frac{C}{r} \right). 
\end{align*}
To obtain a rough upper bound on $C$, note that each $C_i$ can be upper-bounded by $2^d \cdot 2^{d/2} \cdot \Gamma(1+d/2)(3/2)^d$, which leads to
\begin{align*}
C &= \frac{2^{2d-1}}{(2\pi)^{d/2}} \frac{d\pi^{d/2}}{\Gamma(1+d/2)} \cdot \sum_{i=0}^d C_i\\
&\le  \frac{2^{2d-1}}{2^{d/2}} \frac{d}{\Gamma(1+d/2)} \cdot d 2^{d/2} \Gamma(1+ d/2) 3^d\\
&= 2^{2d-1} d^2   3^d.
\end{align*}
Moreover, 
\begin{align*}
C &\geq \frac{2^{2d-1}}{(2\pi)^{d/2}} \frac{d\pi^{d/2}}{\Gamma(1+d/2)} \cdot C_0\\
&= \frac{2^{2d-1}}{(2\pi)^{d/2}} \frac{d\pi^{d/2}}{\Gamma(1+d/2)} \cdot 2^{d/2} \Gamma(1+d/2)\\
&= 2^{2d-1}d.
\end{align*}
This gives $C = e^{\Theta(d)}$. We make no claims about the tightness of this bound with regards to the dimension $d$. As we shall note in Remark~\ref{rem: confusion}, there conflicting arguments whether our analysis can be strengthened to derive a bound with $C = {\Theta(1)}$.

The proof for sets of the form $A_{\bp r}$ follows the same steps as above until equation~\ref{eq: A_tau}, but differs in the definition of $(A^\tau)$, which is now defined as $(A^\tau) \defn A \cap B(\tau+ \sqrt d(r+\delta))$. Using Corollary~\ref{cor: lebesgue_rip}, we may upper-bound $\lambda((A^\tau)_{\bp r + \delta} \setminus (A^\tau)_{\bp r})$ by
\begin{align*}
\lambda((A^\tau)_{\bp r + \delta} \setminus (A^\tau)_{\bp r}) \leq \frac{\omega_d( \tau + \sqrt d(3r/2 + \delta))^d}{r^d} \cdot 2^{2d-1} ((r+\delta)^d - r^d).
\end{align*}
Following the steps above, we arrive at the analogue of equation~\eqref{eq: r_star}:
\begin{align}
\gammacsup(\del A_{\bp r}) &\leq \frac{1}{(2\pi)^{d/2}} \left(2^{2d-1} \omega_d d \right) \cdot \sum_{i=0}^d C_i r^{i-1}, \label{eq: r_star_box}
\end{align}
where $C_i = {d \choose i} 2^{(d-i)/2} \Gamma(1 + (d-i)/2) (\sqrt d)^i (3/2)^i$.

Setting $C \defn \frac{1}{(2\pi)^{d/2}} \left(2^{2d-1} \omega_d d \right) \cdot \sum_{i=0}^d C_i$, it is clear that 
\begin{align*}
\gammacsup(\del A_{\bp r}) \leq \max \left( C, \frac{C}{r} \right).
\end{align*}
The only thing left to check is the dependence of $C$ on the dimension $d$. Using the approximation $\Gamma(k+1) \approx e^{k\log k + \Theta(k)}$, we see that
\begin{align*}
C &= \frac{1}{(2\pi)^{d/2}} \left(2^{2d-1} \omega_d d \right) \cdot \sum_{i=0}^d C_i\\
&= \frac{e^{\Theta(d)}}{d^{d/2}} \sum_{i=0}^d (d-i)^{(d-i)/2} d^{i/2}\\
&= e^{\Theta(d)} \sum_{i=0}^d \left(\frac{i}{d}\right)^{i/2}\\
&= e^{\Theta(d)}.
\end{align*}
\end{proof}

\begin{remark}
The $1/r$ dependence of the Gaussian surface area upper bound is optimal. To see this, let the set $S^r$ be a maximal packing arrangement of radius-$r$ balls in $B_d(1)$, and let $N^r$ be the number of balls that are packed. Clearly, the Gaussian surface area of $S_r$ is at most $N^r \Omega_d r^{d-1} \frac{1}{(2\pi)^{d/2}} e^{-1/2} = C_d M^r r^{d-1}$, where $C_d$ is a dimension-dependent constant. It is now easy to check that $N^r = \Omega(1/r^{d})$, so that $\gamma(\del S_r) = \Omega(1/r)$. A similar argument works for $\gammacsup(\del A_{\bp r})$ as well.
\end{remark}

\begin{remark}
We fixed $r \le 1$ in our proof, which may seem arbitrary. Indeed, a tighter bound on the constant $C$ can be obtained if $r \leq r^*$, where $r^*$ minimizes the bound in expression~\eqref{eq: r_star}. However, by examining the coefficient $C_1$, it may be verified that the tighter constant is still $e^{\Theta(d)}$.
\end{remark}


\begin{remark}\label{rem: sigma}
For the scaled Gaussian distribution $\cN(0, \sigma^2 I_d)$, the Gaussian surface area upper bound changes to $\max\left(\frac{C}{\sigma}, \frac{C}{r} \right)$, where $C$ is as in Theorem~\ref{thm: gaussian_rip}. 
\end{remark}

\begin{remark}\label{rem: mixture}
Theorem~\ref{thm: gaussian_rip} also holds for distributions that have been smoothed by convolving with the Gaussian density. Specifically, if $\mu = \tilde \mu \star \cN(0, \sigma^2 I_d)$ then the $\mu$-surface area of $r$-parallel sets is upper-bounded by $\max\left(\frac{C}{\sigma}, \frac{C}{r} \right)$.
\end{remark}


\section{Reverse Brunn-Minkowski and entropy power inequalities}\label{sec: rev_bmi_epi}

The Brunn-Minkowski inequality~\cite{Gar02} provides a lower bound on the volume of the Minkowski sum of two measurable sets $K$ and $L$ in terms of the volumes of $K$ and $L$. The lower bound is as follows:
\begin{align*}
\lambda(K \oplus L)^{1/d} \geq \lambda(K)^{1/d} + \lambda(L)^{1/d},
\end{align*}
with equality if and only if $K$ and $L$ are homothetic convex bodies from which sets of measure zero have been removed. In general, it is not possible to upper bound $\lambda(K \oplus L)$ in terms of $\lambda(K)$ and $\lambda(L)$. It is easy to construct examples where $\lambda(K) = \lambda(L) = 0$, but $\lambda(K \oplus L)$ is arbitrarily large -- even when $K$ and $L$ are convex sets. Under the convexity assumption, Milman's reverse Brunn-Minkowski inequality~\cite{Mil86, MilPaj00} shows the existence of volume-preserving linear transformations $T_1$ and $T_2$ such that 
\begin{align*}
\lambda(T_1K \oplus T_2 L)^{1/d} \leq C(\lambda(K)^{1/d} + \lambda(L)^{1/d}),
\end{align*}
where $C$ is a dimension-independent constant. Milman's reverse Brunn-Minkowski inequality is a deep result in convex geometry and the local theory of Banach spaces~\cite{Pis99}. 

The entropy power inequality~\cite{Bla65} is the information theoretic analogue of the Bruno-Minkowski inequality. For independent random variables $X$ and $Y$ on $\real^d$ with well-defined differential entropies, the entropy power inequality asserts that
\begin{align*}
e^{2h(X+Y) \over d} \geq e^{2h(X) \over d} + e^{2h(Y) \over d}, 
\end{align*}
with equality if and only if $X$ and $Y$ are Gaussian with proportional covariance matrices. Like the Brunn-Minkowski inequality, the entropy power inequality cannot be reversed in general. Indeed, we may construct examples where $h(X) = h(Y) = -\infty$, but $h(X+Y)$ is arbitrarily large. Inspired by Milman's reverse Brunn-Minkowski inequality, Bobkov and Madiman~\cite{BobMad12} established a reverse entropy power inequality for log-concave random variables\footnote{Bobkov and Madiman proved a more general result for convex (also called hyperbolic) measures, of which log-concave measures comprise a special case.}. If $X$ and $Y$ are independent log-concave random variables, then there exist volume (and therefore entropy) preserving linear transformations $T_1$ and $T_2$ such that 
\begin{align*}
h(T_1X+ T_2Y) \leq C( h(X) + h(Y)),
\end{align*}
where $C$ is a dimension-independent constant. 

As noted above, one of the reasons the Brunn-Minkowski inequality and the entropy power inequality cannot be reversed in general is the ability to use sets of zero volume and random variables with $-\infty$ entropies to construct examples that demonstrate the futility of such a reversal. A natural fix would be to impose some regularity conditions that rule out these problematic examples. In this section, we show that the class of $r$-parallel sets satisfies a version of the reverse Burn-Minkowski inequality, and the class of $r$-smoothed measures (measures obtained by convolving an arbitrary measure with a scaled-standard normal measure) satisfies a version of the reverse entropy power inequality. We will detail results for parallel sets of the form $A_{\bc r}$, which shall be denoted by $A_r$ for simplicity. Analogous results for parallel sets of the form $A_{\bp r}$ may be derived by making minor changes to the proofs presented here. We note that the inequalities presented in this section are largely a consequence of the inherent regularity of parallel sets and smoothed distributions, which is contrast to the more fundamental results in Milman~\cite{Mil86} and Bobkov and Madiman~\cite{BobMad12}.

\subsection{Upper bounds for $\lambda(K_r \oplus L_r)$}
 
\begin{theorem}\label{thm: rev_bmi}
For $r>0$, let $K_r$ and $L_r$ be arbitrary $r$-parallel sets in $\real^d$. Then the following inequality holds:
\begin{align*}
\lambda(K_r \oplus L_r) \leq \lambda(K_r) \lambda(L_r) C(d, r),
\end{align*}
where $C(d, r) = {2^{4d} \over \omega_d r^d}$.
\end{theorem}
\begin{proof}
It is enough to prove the result for $K_r, L_r$ being finite unions of balls, since the general result will follow via a continuity argument as in the proof of Theorem~\ref{thm: lebesgue_rip_volume}. Let $K = \{x_i, i \in [N]\}$ and $L = \{y_j, j \in [M] \}$. Consider a maximal $r$-packing $\{\widehat x_i, i \in [\widehat N]\}$ of $K$ and a maximal $r$-packing $\{\widehat y_j, j \in [\widehat M]\}$ of $L$. For $i \in [\widehat N]$ and $j \in [\widehat M]$, define
\begin{align*}
&(\widehat K)^i_r = \cup_{d(x_k, \widehat x_i) \leq r} B(x_k; r), \quad{ and }\\
&(\widehat L)^j_r = \cup_{d(y_\ell, \widehat y_j) \leq r} B(y_\ell; r).
\end{align*}
Clearly,
\begin{align*}
K_r \oplus L_r = \cup_{i \in [\widehat N], j \in [\widehat M]} (\widehat K)^i_r \oplus (\widehat L)^j_r.
\end{align*}
Since $(\widehat K)^i_r \subseteq B(\widehat x_i; 2r)$, and $(\widehat L)^j_r \subseteq B(\widehat y_j; 2r)$, we have 
\begin{align*}
\lambda(K_r \oplus L_r) &\leq \sum_{i \in [\widehat N], j \in [\widehat M]} \lambda \left( (\widehat K)^i_r \oplus (\widehat L)^j_r \right)\\
&\leq \widehat N \widehat M \lambda(B(\widehat x_i + \widehat y_j; 4r))\\
&\leq \frac{\lambda(K_r)}{\omega_d (r/2)^d} \cdot \frac{\lambda(L_r)}{\omega_d (r/2)^d}\cdot \omega_d (4r)^d\\
&= \lambda(K_r) \lambda(L_r) \cdot \frac{2^{4d}}{\omega_d r^d}.
\end{align*}
\begin{remark}
If $K_r$ and $L_r$ are indeed unions of balls such that the balls $B(x_i+y_j; 2r)$ for $i \in [N]$ and $j \in M$ are disjoint, then it is easy to see that 
\begin{align*}
\lambda(K_r \oplus L_r) = N \cdot M \cdot \omega_d (2r)^d = \lambda(K_r) \lambda(L_r) \cdot \frac{2^{d}}{\omega_d r^d}.
\end{align*}
Thus, the constant $C(d,r)$ has the correct dependence on $r$ and the dependence on $d$ is essentially tight. 
\end{remark}

\end{proof}

\subsection{Upper bounds for $h(X_r + Y_r)$}
Theorem~\ref{thm: rev_bmi} begs the question whether an analogous inequality for entropy holds in information theory. We answer the question in the affirmative. Recall that an $r$-smooth random variable is obtained by convolving the distribution of an arbitrary random variable with the Gaussian distribution $\cN(0, r I)$.
\begin{theorem}\label{thm: rev_epi}
Let $r >0$ and $X_r$ and $Y_r$ be independent $r$-smooth random variables in $\real^d$ with well-defined entropies $h(X_r)$ and $h(Y_r)$ and finite second moments. Then the following inequality holds:
\begin{align*}
h(X_r \oplus Y_r) \leq h(X_r) + h(Y_r) + C(d, r),
\end{align*}
where $C(d, r) = - \frac{d}{2}\log (\pi r)$.
\end{theorem}

We make a few remarks before proving Theorem~\ref{thm: rev_epi}. Suppose we write $X_r+Y_r$ as $X_r + Y+ \sqrt r Z_2$ (or $X + Y_r+ \sqrt r Z_1$) where $Z_1, Z_2 \sim \cN(0, I)$ are such that $(X, Y, Z_1, Z_2)$ are mutually independent. One way to bound the entropy of $X_r+Y_r$ is by using the concavity of entropy along the heat equation~\cite{Cos85}:
\begin{align*}
h(X_r +Y + \sqrt r Z_2) \leq h(X_r +Y) + \frac{J(X_r+Y)}{2} \cdot r.
\end{align*}
One may further upper bound the Fisher information term using $J(X_r + Y) \leq J(\sqrt r Z_2) = d/r$. The interesting aspect of Theorem~\ref{thm: rev_epi} is bounding $h(X_r+Y_r)$ by entropies of $X_r$ and $Y_r$, which requires using the smoothness properties of $X_r$ and $Y_r$. Our proof is inspired by a proof in Bobkov and Marsiglietti~\cite{BobMar20}: Lemma 5.1 in~\cite{BobMar20} (which has also previously appeared in Wang and Madiman~\cite{WanMad14} and Melbourne, Talukdar, Bhaban, Madiman, and Salapaka~\cite{MelEtal18}) states that if $X$ is a discrete random variable and $Y$ is a continuous random variable, then 
\begin{align*}
h(X+Y) \leq H(X) + h(Y).
\end{align*}
\begin{proof}
We prove the result for when $X$ and $Y$ are supported on a finite set of points. If $X$ and $Y$ are continuous, we may consider a sequence of discrete distributions $\widehat X_i \stackrel{d} \to X$ and $\widehat Y_i \stackrel{d} \to Y$. The convergence of $h(\widehat X_i + \sqrt r Z_1) \to h(X + \sqrt r Z_1)$ (and same for the $\widehat Y_i$ sequence) can be concluded using Geng and Nair~\cite[Proposition 18]{GengNair14}.

Denote the standard normal distribution $\cN(0, I)$ by $g(\cdot)$, and $\cN(0, rI)$ by $g_r(\cdot)$. Let $\supp(X) = \{x_1, \dots, x_N\}$ and $\supp(Y) = \{y_1, \dots, y_M\}$ for $N, M \geq 1$. We shall use the shorthand $p_i \defn \prob(X = x_i)$ and $q_j \defn \prob(Y = y_j)$. The distributions of $X_r$ and $Y_r$, denoted by $p$ and $q$, are given by
\begin{align*}
p(x) &= \sum_{i=1}^N p_i g_r(x-x_i), \quad \text{ and}\\
q(y) &= \sum_{j=1}^M q_j g_r(y-y_j).
\end{align*}

The main lemma we use for this proof is the following:
\begin{lemma}
$(p\star q)(t) \geq e^{-C(d, r)} p(\tau)q(t-\tau)$ for every $t, \tau \in \real^d$, where $C(d, r) = - \frac{d}{2}\log (\pi r).$
\end{lemma}
\begin{proof}
The distribution of $X_r+Y_r$ is given by the $p \star q$, which is given by
\begin{align*}
(p \star q)(t) &= \sum_{i=1}^N \sum_{j=1}^M p_i q_j \left(g_r(x-x_i) \star g_r(y-y_r)\right)(t)\\
&= \sum_{i=1}^N \sum_{j=1}^M p_i q_j g_{2r}(t-(x_i+y_j)).
\end{align*}
For a fixed $\tau \in \real^d$, the product $p(\tau)q(t-\tau)$ is given by
\begin{align*}
p(\tau)q(t-\tau) &= \sum_{i=1}^N p_i g_r(\tau-x_i) \cdot \sum_{j=1}^M q_j g_r(t-\tau-y_j)\\
&= \sum_{i=1}^N \sum_{j=1}^M p_i q_j g_r(\tau-x_i)g_r(t-\tau-y_j).
\end{align*}
Thus, it is necessary and sufficient to prove that for any $a, b \in \real^d$
\begin{align*}
g_{2r}(a+b) \geq e^{-C(d, r)} g_r(a)g_r(b).
\end{align*}
Observe that
\begin{align*}
\frac{g_{2r}(a+b)}{g_r(a)g_r(b)} &= (\pi r)^{d/2} \exp \left( \frac{a^2-2ab+b^2}{4r}\right)\\
&\geq (\pi r)^{d/2}\\
&= e^{-C(d, r)},
\end{align*}
where $C(d, r) = - \frac{d}{2}\log (\pi r).$
\end{proof}

With the above lemma in hand, we conclude the proof as follows:
\begin{align*}
- h(p \star q) &= \int_{\real^d} (p\star q)(t) \log (p\star q)(t) dt\\
%
%
&= \int_{\real^d \times \real^d} p(\tau)q(t-\tau) \log ((p\star q)(t)) d\tau dt\\
&\geq \int_{\real^d \times \real^d} p(\tau)q(t-\tau) \log (e^{-C(d, r)} p(\tau)q(t-\tau)) d\tau dt\\
&= - h(p) - h(q) - C(d, r).
\end{align*}
This proves the theorem.
\end{proof}

\begin{remark}
Notice that the upper bound becomes larger for smaller values of $r$. It is also not hard to see that the dependence of $C(d, r)$ on $d$ and $r$ is essentially tight. In the above proof, if the points $\{x_i\}_{i \in [N]}$ and $\{y_j\}_{j \in [M]}$ are all far away from each other, then 
\begin{align*}
h(X_r) &\approx H(p_1, \dots, p_k) + \frac{d}{2} \log (2\pi e r),\\
h(Y_r) &\approx H(q_1, \dots, q_k) + \frac{d}{2} \log (2\pi e r), \quad \text{ and}\\
h(X_r + Y_r) &\approx H(p_1, \dots, p_k)+H(q_1, \dots, q_k)+\frac{d}{2} \log(4\pi e r).
\end{align*}
This gives
\begin{align*}
h(X_r + Y_r) \approx h(X_r) + h(Y_r) - \frac{d}{2} \log (\pi e r).
\end{align*}

\end{remark}


\section{Applications to machine learning}\label{sec: ml}

We now describe applications of the preceding results to problems in machine learning. The results in this section are applicable to parallel sets with respect to $B$ and $C$. To simplify notation, we state our results using the notation $A_r$ which may be replaced by $A_{\bc r}$ or $A_{\bp r}$ as desired.

\subsection{Computational complexity for learning parallel sets}

The sample complexity of learning a class $\cC$ of Boolean functions on $\real^d$ under an unknown distribution is characterized by the Vapnik-Chervonenkis (VC) dimension of $\cC$~\cite{ShaBen14}. It is not hard to check that the VC dimension of indicator functions on $r$-parallel sets is infinite (even if we consider ``bounded'' $r$-parallel sets within a ball $B(R)$ for $R>r$), so this class is not learnable without making some assumptions on the data distribution. Klivans et al.~\cite{KliEtal08} suggested the Gaussian data distribution as a natural setting in which to study the learnability of indicator functions of subsets of $\real^d$. In particular, the authors examined the computational complexity of learning such functions for a variety of subsets, such as halfspaces, convex sets, Euclidean balls, and intersections of halfspaces. The authors proposed the Gaussian surface area of a set as a useful ``complexity measure'' for determining the difficulty of learning, and provided three reasons for doing so: (1) Every measurable set can be assigned a complexity measure; (2) it is a natural geometric notion; and (3) sets with ``wiggly'' boundaries are harder to learn, which is captured by their larger Gaussian surface area. The main result from Klivans et al.\ is as follows:

\begin{theorem}[Theorems 9, 10, and 15 from Klivans et al.~\cite{KliEtal08}]\label{thm: klivans}
Let $\cC$ be a class of measurable sets in $\real^d$ such that the Gaussian surface area of all sets in $\cC$ is upper-bounded by $S$. We shall denote the set of indicator functions of sets in $\cC$ by $\cC$, as well. Under the standard Gaussian distribution, the following results hold for learning $\cC$ up to an accuracy of $\epsilon$ and confidence of $1-\delta$:
\begin{enumerate}
\item
\textbf{Agnostic learning:} There exists an algorithm that runs in time $\text{poly}\left(d^{O(S^2/\epsilon^4)}, \frac{1}{\epsilon}, \log \frac{1}{\delta} \right)$ and agnostically learns $\cC$.
\item
\textbf{PAC learning:} There exists an algorithm that runs in time $\text{poly}\left(d^{O(S^2/\epsilon^2)}, \frac{1}{\epsilon}, \log \frac{1}{\delta} \right)$ and PAC learns $\cC$.
\end{enumerate}
\end{theorem}

An example of an application of the above result is for learning convex sets: Klivans et al.\ showed that, the class of all convex sets---despite having infinite VC dimension---is efficiently learnable under the Gaussian distribution, by exploiting the fact that the Gaussian surface areas of convex sets in $\real^d$ is bounded above by $\Theta(d^{1/4})$. If one is able to bound the Gaussian surface areas of sets in $\cC$, then Theorem~\ref{thm: klivans} may be directly applied to bound the computational complexity of learning $\cC$. Since Theorem~\ref{thm: gaussian_rip} provides bounds on the Gaussian surface areas of $r$-parallel sets, we may directly apply Theorem~\ref{thm: klivans} to conclude the following result concerning the computational complexity of learning $r$-parallel sets:

\begin{theorem}\label{thm: klivans_parallel}
Let $r > 0$. Let $\cC_r = \{A_r ~\mid~ A \subset \real^d ~\text{is measurable}\}$. Under the standard Gaussian distribution, the following results hold for learning $\cC_r$ up to an accuracy of $\epsilon$ and confidence of $1-\delta$:
\begin{enumerate}
\item
\textbf{Agnostic learning:} There exists an algorithm that runs in time
\begin{equation*}
\text{poly}\left(d^{O(\max(C^2, C^2/r^2)/\epsilon^4)}, \frac{1}{\epsilon}, \log \frac{1}{\delta} \right)
\end{equation*}
and agnostically learns $\cC$.
\item
\textbf{PAC learning:} There exists an algorithm that runs in time
\begin{equation*}
\text{poly}\left(d^{O(\max(C^2, C^2/r^2)/\epsilon^2)}, \frac{1}{\epsilon}, \log \frac{1}{\delta} \right)
\end{equation*}
and PAC learns $\cC$.
\end{enumerate}
\end{theorem}

Klivans et al.\ noted that for two sets $K_1$ and $K_2$, we have the inequality $\gamma(\del(K_1 \cup K_2)) \leq \gamma(\del K_1) + \gamma(\del K_2)$. They also showed that the Gaussian surface areas of Euclidean balls (of any radius) are upper-bounded by a constant. Applying this result to the union of balls, we may derive upper bounds on the computational complexity of learning a union of $O(1)$ ball. However, since parallel sets are the union of (possibly) uncountably many balls, the results from Klivans et al.\ cannot be applied directly. Observe also that as $r$ decreases, the boundaries of sets in $\cC_r$ become more ``wiggly'', and the increased difficulty of learning is reflected in the larger exponent $C^2/r^2$. Lastly, our analysis reveals that $C = e^{\theta(d)}$, so for a fixed $r < 1$, the exponent of $d$ in the learning time bounds is $d^{\exp(\Theta(d))/r^2\epsilon^4}$ for agnostic learning and $d^{\exp(\Theta(d))/r^2\epsilon^2}$ for PAC learning. As noted earlier, it may be possible to derive a stronger upper bound where $C = \Theta(1)$; if so, the corresponding computational complexity bounds would also be strengthened.

\begin{remark}\label{rem: confusion}
One reason to believe $C$ could be made $\Theta(1)$ is as follows. By taking $r \to \infty$ in Theorem~\ref{thm: gaussian_rip}, the upper bound is simply $C$. Intuitively, a parallel set when $r$ is very large resembles a halfspace whose Gaussian surface area is known to be bounded by $\Theta(1)$. A reason to believe $C$ cannot be made $\Theta(1)$ is because it would lead to the surprising result that $1$-parallel sets are essentially as hard to learn as halfspaces. This runs counter to intuition since $1$-parallel sets appear to be far more expressive than halfspaces. 
\end{remark}

\subsection{Sample complexity for estimating robust risk}

Adversarial machine learning has been the focus of much research in the recent past, owing to the observed fragility of deep neural networks under adversarial perturbations. A brief description of the underlying mathematical problem phrased in the language of hypothesis testing is provided below.

\subsubsection{Problem setting and background}

Consider two equally likely hypotheses, denoted by $\{0,1\}$. For $i \in \{0,1\}$, under hypothesis $i$, a sample drawn from distribution $\mu_i$ is observed. To minimize the error probability, it is well known that the optimal decision rule is the maximum likelihood rule and the resulting error (called the Bayes risk) is given by $\frac{1- d_{TV}(\mu_0, \mu_1)}{2}$, where $d_{TV}$ is the total variation distance. Hypothesis testing under adversarial contamination considers an identical setting with one modification: The adversary is allowed to arbitrarily perturb the observed sample within a Euclidean ball of a certain radius, say $r > 0$. The radius $r$ is the adversary's budget. 

Finding the optimal decision region for hypothesis testing with an adversary has been studied recently in Bhagoji, Cullina, and Mittal~\cite{BhaEtal19} and Pydi and Jog~\cite{PydiJog20}. Suppose $A$ is the (measurable) set where hypothesis 1 is declared. Then the robust risk for this decision region is given by
\begin{align*}
\cE(A) &= \frac{\mu_0(A_r) + \mu_1((A^c)_r)}{2}\\
&= \frac{\mu_0(A_r) + 1 -  \mu_1(((A^c)_r))^c}{2}\\
&\stackrel{(a)} = \frac{1}{2} - \frac{\mu_1(A_{-r}) - \mu_0(A_r)}{2},
\end{align*}
where in $(a)$, we use the notation $A_{-r} = ((A^c)_r)^c$. As shown in Pydi and Jog~\cite{PydiJog20}, the optimal robust risk may also be expressed as 
\begin{align*}
\cE^* &= \frac{1}{2} - \sup_A \frac{\mu_1(A_{-r}) - \mu_0(A_r)}{2}\\
&= \frac{1}{2} - \sup_A \frac{\mu_1(A) - \mu_0(A_{2r})}{2}.
\end{align*}
The main result of Bhagoji et al.\ and Pydi and Jog connects the optimal robust risk to an optimal transport cost between the two data distributions. To be precise, the following result was established:

\begin{theorem}[Bhagoji et al.~\cite{BhaEtal19} and Pydi and Jog~\cite{PydiJog20}] \label{thm: muni}
Define the cost function $c_r: \real^d \times \real^d \to \real$ as
\begin{align*}
c(x,y) = \mathbbm 1\{ d(x,y) > 2r \},
\end{align*}
where $d$ is the usual Euclidean distance. Define the optimal transport cost $D_r(\mu_0, \mu_1)$ between two distributions $\mu_0$ and $\mu_1$ on $\real^d$ as
\begin{align*}
D_r(\mu_0, \mu_1) = \inf_{X \sim \mu_0, Y \sim \mu_1} \E c_r(X,Y),
\end{align*}
where the infimum is taken over all couplings of $X$ and $Y$ with marginals $\mu_0$ and $\mu_1$, respectively. Then for an adversarial budget $r$, the optimal robust risk for a binary hypothesis testing problem with equal priors and data distributions $\mu_0$ and $\mu_1$ satisfies the equality
\begin{align*}
\cE^* = \frac{1 - D_r(\mu_0, \mu_1)}{2}.
\end{align*}
\end{theorem}
The above result follows from Strassen's theorem~\cite{Vil03}, which gives the equality
\begin{align*}
\sup_A \frac{\mu_1(A) - \mu_0(A_{2r})}{2} = D_r(\mu_0, \mu_1).
\end{align*}

\subsubsection{Estimating $D_r(\mu_0, \mu_1)$}

Theorem~\ref{thm: muni} is useful because it provides a fundamental lower bound for robust risk that holds for all hypothesis testing rules. One may evaluate a testing rule based on the closeness of its performance to this optimal value. However, this is not possible in practice, since the data distributions $\mu_0$ and $\mu_1$ are unknown; one only has access to the ``empirical distribution'' derived from a data set composed of independent draws from the data distribution. 

In Bhagoji et al.\ and Pydi and Jog, the authors calculate $D_r(\cdot, \cdot)$ between the empirical distributions (based on a finite data set), instead. However, neither work addresses the proximity of the empirically calculated $D_r$ to the true $D_r$. Indeed, it is not even clear if $D_r$ calculated from the empirical distribution is a consistent estimator of the true $D_r$. 

Some intuition about $D_r$ can be obtained by observing that when $r=0$, it equals the total variation distance. As noted in Pydi and Jog, however, $D_r$ is neither a metric nor a pseudo-metric on the space of probability distributions. Moreover, estimating the total variation distance between $\mu_0$ and $\mu_1$ by calculating the total variation distance between the the empirical distributions is bound to fail, since the latter will always yield a value of 1 for continuous $\mu_0$ and $\mu_1$. Interestingly, this is not the case for $D_r$ when $r > 0$. 

In what follows, we show that under suitable smoothness conditions, the plug-in estimator is a consistent estimator of the true $D_r$. We also provide bounds on the number of samples necessary to approximate $D_r(\mu_0, \mu_1)$ up to an error of $\epsilon$ with a probability of $1-\delta$. The main technical ingredient is the reverse Gaussian isoperimetric inequality from Theorem~\ref{thm: gaussian_rip}.

\subsubsection{Sample complexity bounds for estimating $D_r(\mu_0, \mu_1)$}
We make the following assumptions on $\mu_0$ and $\mu_1$:
\begin{itemize}
\item[(A1)] Each $\mu_i$ is a Gaussian-smoothed version of some $\tilde \mu_i$; i.e., $\mu_i = \tilde \mu_i \star \cN(0, \sigma^2 I_d)$.
\item[(A2)] Each $\tilde \mu_i$ has bounded support on $\real^d$.
\end{itemize}
Assumption $(a)$ is easy to satisfy in practice by simply adding Gaussian noise to the observed data samples. Assumption $(b)$ makes our analysis simpler, but we note that it can be considerably relaxed. We use the notation $\mu_i^n$ for the empirical distribution with $n$ samples, and use the random variables $X_0 \sim \mu_0$, $X_1 \sim \mu_1$, $X_0^n \sim \mu_0^n$, and $X_1^n \sim \mu_1^n$. If $X \sim \mu_X$ and $Y \sim \mu_Y$, we shall use the notation $D_r(\mu_X, \mu_Y)$ and $D_r(X, Y)$ interchangeably. We denote $C(\sigma, r) \defn \max \left(\frac{C}{\sigma}, \frac{C}{r}\right)$, where $C$ is as in Theorem~\ref{thm: gaussian_rip}. 

\begin{lemma}\label{lemma: adv1}[Corollary 3.1 from Pydi and Jog~\cite{PydiJog20}] 
The following inequality holds: 
\begin{align*}
D_r(\mu_0, \mu_1) \leq \frac{W_1(\mu_0, \mu_1)}{2r},
\end{align*}
where $W_1$ is the 1-Wasserstein distance.
\end{lemma}
Lemma~\ref{lemma: adv1} is a straightforward consequence of applying Markov's inequality to the equality $D_r(X, Y) = \inf_{\Pi(X,Y)} \prob(d(X,Y) > 2r)$, where $\Pi(X,Y)$ is the set of couplings of $X$ and $Y$.

\begin{lemma}\label{lemma:  adv2}
Let $\eta \in (0,r/3)$. The following inequalities hold:
\begin{align*}
D_{r+2\eta}(X_0, X_1) &\leq D_{r}(X_0^n, X_1^n) + D_\eta(X_0, X_0^n) + D_\eta(X_1, X_1^n), \quad \text{ and }\\
D_{r-2\eta}(X_0, X_1)  &\geq D_{r}(X_0^n, X_1^n) - D_\eta(X_0, X_0^n) - D_\eta(X_1, X_1^n).
\end{align*}
\end{lemma}

\begin{proof}
Consider a coupling of $(X_0, X_1, X_0^n, X_1^n)$ such that the Markov chain $X_0 \to X_0^n \to X_1^n \to X_1$ holds. The joint distributions of adjacent links in the chain are as follows:  $(X_0, X_0^n) \sim \pi_0$, which is optimal for $D_\eta$; $(X_1, X_1^n) \sim \pi_1$, which is optimal for $D_\eta$; and $(X_0^n, X_1^n) \sim \pi_{01}^n$, which is optimal for $D_{r}$. The Markov chain induces a coupling on $(X_0, X_1)$ that is not necessarily optimal for the $D_{r+2\eta}$ cost. This means that
\begin{align*}
D_{r+2\eta}(X_0, X_1) &\leq \prob(d(X_0,X_1) > 2r + 4\eta)\\
&\leq \prob(d(X_0, X_0^n) > 2\eta) + \prob(d(X_1, X_1^n) > 2\eta) + \prob(d(X_0^n, X_1^n) > 2r)\\
&= D_\eta(X_0, X_0^n) + D_\eta(X_1, X_1^n) + D_{r}(X_0^n, X_1^n).
\end{align*}

To obtain the second inequality, consider a different coupling between $(X_0, X_1, X_0^n, X_1^n)$ such that the Markov chain $X_0^n \to X_0 \to X_1 \to X_1^n$ holds. The joint distributions of adjacent links in the chain are as follows:  $(X_0, X_0^n) \sim \pi_0$, which is optimal for $D_\eta$; $(X_1, X_1^n) \sim \pi_1$, which is optimal for $D_\eta$; and $(X_0, X_1) \sim \pi_{01}$, which is optimal for $D_{r-2\eta}$. The Markov chain induces a coupling on $(X_0^n, X_1^n)$ that is not necessarily optimal for the $D_{r}$ cost. This means that
\begin{align*}
D_{r}(X_0^n, X_1^n) &\leq \prob(d(X_0^n,X_1^n) > 2r)\\
&\leq \prob(d(X_0, X_0^n) > 2\eta) + \prob(d(X_1, X_1^n) > 2\eta) + \prob(d(X_0, X_1) > 2r-4\eta)\\
&= D_\eta(X_0, X_0^n) + D_\eta(X_1, X_1^n) + D_{r-2\eta}(X_0, X_1).
\end{align*}
\end{proof}

\begin{lemma}\label{lemma: adv3}
Let $0 < r_1 < r_2$. Then the following inequality holds:
\begin{align*}
0 \leq D_{r_1}(\mu_0, \mu_1) - D_{r_2}(\mu_0, \mu_1) \leq 2C(\sigma, 2r_1)(r_2-r_1).
\end{align*}
\end{lemma}
\begin{proof}
By the definition of $D_r$, it follows immediately that $D_{r_1} (\mu_0, \mu_1) \geq D_{r_2}(\mu_0, \mu_1)$. Let $A^*$ be the set that achieves the equality \footnote{If $A^*$ does not exist, the proof goes through by considering a sequence of sets $(A^*)^n$ such that  $\mu_0((A^*)^n) - \mu_1((A^*)^n_{2r_1}) \stackrel{n \to \infty} \to \sup \mu_0(A) - \mu_1(A_{2r_1})$.}
$$D_{r_1}(\mu_0, \mu_1) = \mu_0(A^*) - \mu_1(A^*_{2r_1}).$$
We have the sequence of inequalities
\begin{align*}
D_{r_2}(\mu_0, \mu_1) &= \sup_A \mu_0(A) - \mu_1(A_{2r_2})\\
&\geq \mu_0(A^*) - \mu_1(A^*_{2r_2})\\
&= \mu_0(A^*) - \mu_1(A^*_{2r_1}) - \mu_1(A^*_{2r_2} \setminus A^*_{2r_1})\\
&= D_{r_1}(\mu_0, \mu_1) - \mu_1(A^*_{2r_2} \setminus A^*_{2r_1}).
\end{align*}
We also have the equality
\begin{align*}
\mu_1(A^*_{2r_2} \setminus A^*_{2r_1}) &= \int_{\tau = 0}^{2(r_2-r_1)} \mu_1(\del A^*_{2r_1+\tau}) d\tau\\
&\stackrel{(a)}\leq 2C(\sigma, 2r_1)(r_2-r_1).
\end{align*}
Here, $(a)$ follows from Remarks~\ref{rem: sigma} and \ref{rem: mixture}.
\end{proof}

\begin{theorem}\label{thm: estimate_dr}
Let $\mu_0$ and $\mu_1$ satisfy the assumptions (A1) and (A2). Also assume that $d \geq 3$. Let $r, \epsilon > 0$ and $\delta \in (0,1)$. Then for $n \geq N_0 = \Theta \left( \frac{C(\sigma, 2r/3)^d \log(1/\delta) } { \epsilon^{2d}}\right)$, the following inequality holds with probability at least $1-\delta$:
\begin{align*}
\abs{D_r(\mu_0, \mu_1) - D_r(\mu_0^n, \mu_1^n)} \leq \epsilon.
\end{align*}
\end{theorem}
\begin{proof}
Let $\eta \in (0, r/3)$. Let $N_0$ be such that for $n \geq N_0$, the following holds with probability $1 - \delta/2$, for $i \in \{0, 1\}$:
\begin{align}\label{eq: w1_bound}
W_1(X_i, X_i^n) \leq 2\eta \cdot \frac{\epsilon}{4}.
\end{align}
Numerous results exist concerning the convergence of the empirical measure in terms of the Wasserstein metric; we use here a result from Fournier and Guillin \cite[Theorem 2]{FouGui15}, which states that for all $N \geq 1$ and all small enough $x$, 
\begin{align*}
\prob(W_1(\mu_i, \mu_i^N) > x) \leq c_0 e^{-c_1 N x^{d}},
\end{align*}
where $c_0$ and $c_1$ are constants that depend on $d$. Substituting $x$ to be $\eta\epsilon/2$ and
\begin{align*}
N  \geq \frac{\log (2/\delta)+ \log c_0}{c_1 (\eta\epsilon/2)^d} =: N_0,
\end{align*}
inequality~\eqref{eq: w1_bound} is satisfied with probability $1-\delta/2$.
When $N \geq N_0$, Lemma~\ref{lemma: adv2} implies that the following bound holds with probability $1-\delta$:
\begin{align*}
D_{r+2\eta}(X_0, X_1) &\leq D_{r}(X_0^n, X_1^n) + D_\eta(X_0, X_0^n) + D_\eta(X_1, X_1^n)\\
&\leq D_{r}(X_0^n, X_1^n) + \frac{W_1(X_0,X_0^n)}{2\eta} + \frac{W_1(X_1,X_1^n)}{2\eta}\\
&\leq D_{r}(X_0^n, X_1^n) + \frac{\epsilon}{2}.
\end{align*}
Lemma~\ref{lemma: adv3} gives the inequality 
\begin{align*}
D_{r}(X_0, X_1) &\leq D_{r+2\eta}(X_0, X_1) + 4\eta C(\sigma, 2r),  
\end{align*}
implying that
\begin{align}\label{eq: dr_end1}
D_{r}(X_0, X_1) \leq D_{r}(X_0^n, X_1^n) + \frac{\epsilon}{2} + 4\eta C(\sigma, 2r).
\end{align}
Similarly, with probability $1-\delta$, 
\begin{align*}
D_{r-2\eta}(X_0, X_1)  &\geq D_{r}(X_0^n, X_1^n) - D_\eta(X_0, X_0^n) - D_\eta(X_1, X_1^n)\\
&\geq D_{r}(X_0^n, X_1^n) - \frac{\epsilon}{2}.
\end{align*}
Lemma~\ref{lemma: adv3} gives the inequality 
\begin{align*}
D_{r}(X_0, X_1) &\geq D_{r-2\eta}(X_0, X_1) - 4\eta C(\sigma, 2(r-2\eta)),  
\end{align*}
implying
\begin{align}\label{eq: dr_end2}
D_{r}(X_0, X_1) \geq D_{r}(X_0^n, X_1^n) - \frac{\epsilon}{2} - 4\eta C(\sigma, 2(r-2\eta)).
\end{align}
Combining inequalities~\eqref{eq: dr_end1} and~\eqref{eq: dr_end2}, we see that with probability $1-\delta$,
\begin{align*}
\abs{D_{r}(X_0, X_1) - D_{r}(X_0^n, X_1^n) } \leq \frac{\epsilon}{2} + 4\eta C(\sigma, 2(r-2\eta)) \leq \frac{\epsilon}{2} + 4\eta C(\sigma, 2r/3).
\end{align*}
Now pick $\eta = \frac{\epsilon}{2C(\sigma, 2r/3)}$ to conclude that
\begin{align*}
\abs{D_{r}(X_0, X_1) - D_{r}(X_0^n, X_1^n) } \leq \epsilon.
\end{align*}
Note that 
\begin{align*}
N_0 &= \frac{\log (2/\delta)+ \log c_0}{c_1 (\eta\epsilon/2)^d}\\
&= \Theta \left( \frac{C(\sigma, 2r/3)^d \log(1/\delta) } { \epsilon^{2d}}\right).
\end{align*}
This concludes the proof.
\end{proof}
\begin{remark}
Observe that smaller values of $r$ correspond to a larger sample size requirement. Also, the smaller the variance $\sigma$ used for smoothing the distributions, the more samples are required. Both observations align with intuition.
\end{remark}

\section{Conclusion and open problems}\label{sec: end}

Convolving with a small Gaussian noise is a common technique used in analysis to smooth probability  distributions. The natural counterpart to such a procedure in geometry is to take the parallel set of any measurable set. It is intuitive that parallel sets are ``more smooth,'' since they cannot have arbitrarily wiggly boundaries. In this paper, we showed that bounded parallel sets in $\real^d$ have bounded Euclidean surface areas and arbitrary parallel sets have bounded Gaussian surface areas. We showed that our reverse isoperimetric inequalities have applications in machine learning. We also established some reverse Brunn-Minkowski and entropy power inequalities that may be of independent interest. We mention a few open problems that are worth exploring.
\begin{openproblem}
The dependence of $C$ on the dimension $d$ in Theorem~\ref{thm: gaussian_rip} is $e^{\Theta(d)}$. Is this dependence tight? If not, what is the right dependence?
\end{openproblem}
An interesting but challenging problem is identifying the Gaussian surface area of an optimally dense packing of unit $\ell_2$- or $\ell_\infty$-balls in $\real^d$, and use this to get a lower bound for the dimension dependence of $C$. 

Our analysis relied heavily on geometric properties of balls and cubes and is therefore restricted to $r$-parallel sets for balls and cubes. For a set $K \in \cK$, if one is able to establish the analogue of Proposition~\ref{prop: puzzle_d_dim} or~\ref{prop: boxes}, then it will be possible to derive Theorems~\ref{thm: lebesgue_rip_volume} and~\ref{thm: gaussian_rip} for sets of the form $A \oplus rK$. We propose the following open problem:
\begin{openproblem}
Let $K \in \cK$. Let $K(x_0;r)$ denote the set $K$ scaled by $r$ and translated so that its center of mass is at $x_0$. Let $A = \{x_i ~|~ i \in [N], x_i \in K(x_0; r)\} \cup \{x_0\}$. Then is it true that $\lambda_K(\del (A \oplus rK)) \leq \lambda_K(\del K(x_0; 2r))$?
\end{openproblem}

One may also consider arbitrary bounded convex sets in addition to sets in $\cK$. A possible approach would be to establish analogs of Theorems~\ref{thm: lebesgue_rip_volume} and \ref{thm: gaussian_rip} for Minkowski smoothing by convex polytopes, and taking the limit to generalize these results for arbitrary convex sets. 

Our final open problem concerns the tightness of the bound in Theorem~\ref{thm: rev_epi}:
\begin{openproblem}
Does the following inequality hold? 
$$h(X_r + Y_r) \leq h(X_r) + h(Y_r) - \frac{d}{2} \log (\pi e r).$$
\end{openproblem}
The reverse entropy power inequality in Section~\ref{sec: rev_bmi_epi} was established for Gaussian-smoothed random variables. It is likely that similar inequalities exist for alternate smoothing procedures, such as exponential smoothing. We leave these problems for future work


\section*{Acknowledgements}

The author is grateful to the National Science Foundation for funding his research through the grants CCF-1907786 and CCF-1942134, and to Ankit Pensia and Muni Sreenivas Pydi for helpful discussions. We also thank Kostiantyn Drach for pointing out the observation in Remark~\ref{rem: kostya}. The figures in Section~\ref{sec: puzzle} were created using GeoGebra, a free online tool for geometry.
  
\bibliographystyle{unsrt}
\bibliography{ref.bib}


\appendix

\end{document}